\documentclass[twoside,a4paper,reqno,11pt]{amsart}
\usepackage{amsfonts, amsbsy, amsmath, amssymb, latexsym}
\usepackage{mathrsfs,array}
\usepackage[top=30mm,right=30mm,bottom=30mm,left=30mm]{geometry}
\usepackage{stmaryrd}
\usepackage{bm}

\usepackage{hyperref}

\newcommand{\cd}{{\rm cd}}
\renewcommand{\ss}{{\rm ss}}

\renewcommand{\a}{\alpha}

\newcommand{\BM}{\mathbb{B}}

\newcommand{\e}{\epsilon}
 
\renewcommand{\l}{\lambda} \renewcommand{\O}{\Omega}

 \renewcommand{\to}{\rightarrow}

 \newcommand{\C}{\mathcal{C}}

\newcommand{\la}{\langle}
\newcommand{\ra}{\rangle}

\newcommand{\leqs}{\leqslant}
\newcommand{\geqs}{\geqslant}

 \newcommand{\vs}{\vspace{3mm}}

\newcommand{\cA}{\mathcal{A}}

\makeatletter
\newcommand{\imod}[1]{\allowbreak\mkern4mu({\operator@font mod}\,\,#1)}
\makeatother

\newtheorem{theorem}{Theorem}
\newtheorem*{conj*}{Conjecture}

\newtheorem{corol}[theorem]{Corollary}

\newtheorem{thm}{Theorem}[section]
\newtheorem{prop}[thm]{Proposition}
\newtheorem{lem}[thm]{Lemma}

\theoremstyle{definition}

\newtheorem{remark}{Remark}

\begin{document}

\author{Timothy C. Burness}
 \address{T.C. Burness, School of Mathematics, University of Bristol, Bristol BS8 1TW, UK}
 \email{t.burness@bristol.ac.uk}

\address{D.R. Heath-Brown, Mathematical Institute, Radcliffe Observatory Quarter, Woodstock Road,
Oxford OX2 6GG, UK}
\email{rhb@maths.ox.ac.uk}

 \author{Martin W. Liebeck}
\address{M.W. Liebeck, Department of Mathematics,
    Imperial College, London SW7 2BZ, UK}
\email{m.liebeck@imperial.ac.uk}

\author{Aner Shalev}
\address{A. Shalev, Institute of Mathematics, Hebrew University, Jerusalem 91904, Israel}
\email{shalev@math.huji.ac.il}

\title{On the length and depth of finite groups}
\dedicatory{\rm With an appendix by D.R. Heath-Brown}

\begin{abstract}
An unrefinable chain of a finite group $G$ is a chain of subgroups $G = G_0 > G_1 > \cdots > G_t = 1$, where each $G_i$ is a maximal subgroup of $G_{i-1}$. The length (respectively, depth) of $G$ is the maximal (respectively, minimal) length of such a chain. We studied the depth of finite simple groups in a previous paper, which included a classification of the simple groups of depth $3$. Here we go much further by determining the finite groups of depth $3$ and $4$. We also obtain several new results on the lengths of finite groups. For example, we classify the simple groups of length at most $9$, which extends earlier work of Janko and Harada from the 1960s, and we use this to describe the structure of arbitrary finite groups of small length. We also present a number-theoretic result of Heath-Brown, which implies that there are infinitely many non-abelian simple groups of length at most $9$.

Finally we study the chain difference of $G$ (namely the length minus the depth). We obtain results on groups with chain difference $1$ and $2$, including a complete classification of the simple groups with chain difference $2$, extending earlier work of Brewster et al. We also derive  a best possible lower bound on the chain ratio (the length divided by the depth) of simple groups, which yields an explicit linear bound on the length of $G/R(G)$ in terms of the chain difference of $G$, where $R(G)$ is the soluble radical of $G$.
\end{abstract}

\subjclass[2010]{Primary 20E32, 20E15; Secondary 20E28}
\date{\today}
\maketitle

\section{Introduction}\label{s:intro}

An \emph{unrefinable} chain of length $t$ of a finite group $G$ is a chain of subgroups
\begin{equation}\label{e:chain}
G = G_0 > G_1 > \cdots > G_{t-1} > G_t=1,
\end{equation}
where each $G_{i}$ is a maximal subgroup of $G_{i-1}$. The \emph{length} of $G$, denoted by $l(G)$, is the maximal length of an unrefinable chain. This notion arises naturally in several different contexts, finding a wide range of applications. For example, Babai \cite{Babai} investigated the length of symmetric groups in relation to the computational complexity of algorithms for finite permutation groups. In a different direction,
Seitz, Solomon and Turull studied the length of finite groups of Lie type in a series of papers in the early 1990s \cite{SST,ST,ST2}, motivated by applications to fixed-point-free automorphisms of finite soluble groups. In fact, the notion predates both the work of Babai and Seitz et al. Indeed,
Janko and Harada studied the simple groups of small length in the 1960s, culminating in Harada's description of the finite simple groups of length at most $7$ in \cite{Harada}.

Given the definition of $l(G)$, it is also natural to consider the \emph{minimal} length of an unrefinable chain for $G$. Following \cite{BLS1}, we call this number the \emph{depth} of $G$, denoted by $\l(G)$. For example, if $G$ is a cyclic group of order $n \geqs 2$, then $\l(G) = \O(n)$, the number of prime divisors of $n$ (counting multiplicities). In particular, $\l(G)=1$ if and only if $G$ has prime order. This notion appears in the earlier work of several authors. For example, in \cite{SW} Shareshian and Woodroofe investigate the length of various chains of subgroups of finite groups $G$ in the context of lattice theory (in their paper, the depth of $G$ is denoted by $\textsf{minmaxl}(G)$). There are also several papers on the so-called \emph{chain difference} ${\rm cd}(G) = l(G) - \l(G)$ of a finite group $G$. For example, a well known theorem of Iwasawa \cite{I} states that ${\rm cd}(G) = 0$ if and only if $G$ is supersoluble. The simple groups $G$ with ${\rm cd}(G)=1$ have been determined by Brewster et al. \cite{BWZ} (also see \cite{HS} and \cite{Pet} for related results).

In \cite{BLS1}, we focus on the depth of finite simple groups. One of the main results is \cite[Theorem 1]{BLS1}, which determines the simple groups of depth $3$ (it is easy to see that $\l(G) \geqs 3$ for every non-abelian simple group $G$); the groups that arise are recorded in  Table \ref{tab:main}. We also show that alternating groups have bounded depth (indeed, $\l(A_n) \leqs 23$ for all $n$, whereas $l(A_n)$ tends to infinity with $n$) and we obtain upper bounds on the depth of each simple group of Lie type. The exact depth of each sporadic simple group is given in \cite[Lemma 3.3]{BLS1}. We refer the reader to \cite{BLS2, BLS3} for results on analogous notions of length and depth for connected algebraic groups over algebraically closed fields and connected compact Lie groups.

\renewcommand{\arraystretch}{1.1}
\begin{table}[h]
\[\begin{array}{lll} \hline\hline
G && \hspace{5mm} \mbox{Conditions} \\ \hline
A_p && \hspace{5mm} \mbox{$p$ and $(p-1)/2$ prime, $p \not\in \{7,11,23\}$} \\
{\rm L}_2(q) && \left\{\begin{array}{l}
\mbox{$(q+1)/(2,q-1)$ or $(q-1)/(2,q-1)$ prime, $q \ne 9$; or} \\
\mbox{$q$ prime and $q \equiv \pm 3, \pm 13 \imod{40}$; or} \\
\mbox{$q=3^k$ with $k \geqs 3$ prime}
\end{array}\right. \\
{\rm L}_n^{\e}(q) && \hspace{5mm} \mbox{$n$ and $\frac{q^n-\e}{(q-\e)\,(n,q-\e)}$ both prime, $n \geqs 3$ and} \\
&& \hspace{5mm} (n,q,\e) \ne (3,4,+), (3,3,-), (3,5,-), (5,2,-) \\
{}^2B_2(q) & & \hspace{5mm} \mbox{$q-1$ prime} \\
{\rm M}_{23},  \; \mathbb{B} && \\ \hline \hline
\end{array}\]
\caption{The simple groups $G$ with $\l(G)=3$}
\label{tab:main}
\end{table}
\renewcommand{\arraystretch}{1}

Our goal in this paper is to extend the depth results in \cite{BLS1} in several different directions, both for simple groups, as well as arbitrary finite groups. We also revisit some of the aforementioned results of Janko and Harada from the 1960s, providing a precise description of the simple groups of small length. In turn, this allows us to describe the structure of arbitrary finite groups of small length and we can use this to classify the simple groups $G$ with $\cd(G)=2$, which extends one of the main results in \cite{BWZ}.

\subsection{Main results on depth}

By a theorem of Shareshian and Woodroofe \cite[Theorem 1.4]{SW}, it follows that $\l(G) \geqs 3$ for every insoluble finite group $G$. Our first main result determines all finite groups of depth $3$. In particular, notice that an obvious consequence of the theorem is that almost simple groups of depth $3$ are simple.

\begin{theorem}\label{depth3}
A finite group $G$ has depth $3$ if and only if either $G$ is soluble of chief length $3$, or $G$ is a simple group as in Table $\ref{tab:main}$.
\end{theorem}

The next result classifies the finite groups of depth $4$. In part (iv) of the statement, a \emph{twisted wreath product} $T \,{\rm twr}_\phi\, S$ for non-abelian simple groups $S,T$ is as defined and studied in \cite{Bad}. The ingredients are a transitive action of $S$ on $k$ points with point stabiliser $S_1$, and a homomorphism $\phi:S_1\to {\rm Aut}(T)$ with image containing ${\rm Inn}(T)$.
Thus $T$ is isomorphic to a proper section of $S$; indeed, the subgroups $C = \phi^{-1}({\rm Inn}(T))$ and $D = {\rm ker}(\phi)\cap C$ satisfy $C/D \cong T$, and $(C,D)$ forms an \emph{$S$-maximal section} of $S$, as defined in \cite[Definition 4.1]{Bad}. Moreover $T \,{\rm twr}_\phi\, S$ is a semidirect product $T^k.S$ having $S$ as a maximal subgroup.

\begin{theorem}\label{depth4}
Suppose $G$ is a finite group of depth $4$. Then one of the following holds, where $p$ is a prime:
\begin{itemize}\addtolength{\itemsep}{0.2\baselineskip}
\item[{\rm (i)}] $G$ is soluble of chief length $4$.
\item[{\rm (ii)}] $G = T\times T$ or $T \times C_p$, where $T$ is simple of depth $3$ (as in Table $\ref{tab:main}$).
\item[{\rm (iii)}] $G = (C_p)^k. T$ (a possibly nonsplit extension), where $T$ is simple of depth $3$, and acts irreducibly on $(C_p)^k$.
\item[{\rm (iv)}] $G = T \,{\rm twr}_\phi\, S$, a twisted wreath product, where $S,T$ are simple, $T$ is a proper section of $S$, and $S$ has depth $3$.
\item[{\rm (v)}] $G$ is quasisimple and $Z(G) = C_p$.
\item[{\rm (vi)}] $G$ is almost simple with socle $T$, and $G/T$ is either $1$ or $C_p$.
\end{itemize}
\end{theorem}

\begin{remark}\label{r:1}
Let us comment on the groups arising in parts (i)--(iv) of Theorem \ref{depth4}.
\begin{itemize}\addtolength{\itemsep}{0.2\baselineskip}
\item[{\rm (a)}] By a theorem of Kohler \cite[Theorem 2]{K}, the depth of a soluble group is equal to the length of a chief series, so every group arising in part (i) does indeed have depth $4$.
\item[{\rm (b)}] In parts (ii) and (iv), note that $G$ has a simple maximal subgroup of depth $3$. In particular, these groups have depth $4$, and the examples arising in (ii) can be listed by inspecting Table \ref{tab:main}. In (iii), for split extensions this is also the case; for nonsplit, $G$ must have a maximal subgroup $M = (C_p)^k.M_0$ with $M_0$ maximal in $T$ acting irreducibly on $(C_p)^k$ (so $\l(M_0)=2$ and $\l(M)=3$). See Section \ref{ss:exd4} for further comments on the nonsplit groups $G = (C_p)^k.T$ arising in part (iii).
\item[{\rm (c)}] In (iv), the possibilities for $S$ are recorded in Table \ref{tab:main}. Every proper non-abelian simple section $T$ of $S$ can occur. The simple sections of the groups $A_p$ and ${\rm L}_n^\e(q)$ in Table \ref{tab:main} cannot be listed. However, the proper simple sections of the other groups in the table can be determined: for ${\rm L}_2(q)$ they are $A_5$, ${\rm L}_2(q_0)$ (with $q=q_0^k$ and $q_0 \geqs 4$); for $^2\!B_2(q)$ there are none (since $q-1$ is prime); and those for ${\rm M}_{23}$ and $\BM$ can be listed (up to isomorphism) using \cite{Atlas}. 
\end{itemize}
\end{remark}

Next consider case (v) in Theorem \ref{depth4}, so $G$ is quasisimple with $Z(G)=C_p$ ($p$ prime) and let $T=G/Z(G)$, a simple group. Here $\l(G) = \l(T)+1$ (see Lemma \ref{add}(iii)), so $\l(G)=4$ if and only if $T$ is one of the groups in Table \ref{tab:main}. In particular, by considering the Schur multipliers of the relevant simple groups, we obtain the following result.

\begin{theorem}\label{depth4q}
Let $G$ be a quasisimple group with nontrivial centre. Then $\l(G)=4$ if and only if $G$ is one of the groups in Table $\ref{tab:quasi}$.
\end{theorem}

\renewcommand{\arraystretch}{1.1}
\begin{table}[h]
\[\begin{array}{ll} \hline\hline
G & \hspace{5mm} \mbox{Conditions} \\ \hline
2.A_p & \hspace{5mm} \mbox{$p$ and $(p-1)/2$ prime, $p \not\in \{7,11,23\}$} \\
{\rm SL}_{n}^{\e}(q) & \hspace{5mm} \mbox{$n = (n,q-\e)$ and $\frac{q^n-\e}{n(q-\e)}$ both prime, $n \geqs 3$ and} \\
& \hspace{5mm} (n,q,\e) \ne (3,4,+), (3,5,-) \\
{\rm SL}_{2}(q) & \left\{\begin{array}{l}
\mbox{$(q+1)/2$ or $(q-1)/2$ prime, $q \ne 9$; or} \\
\mbox{$q$ prime and $q \equiv \pm 3, \pm 13 \imod{40}$; or} \\
\mbox{$q=3^k$ with $k \geqs 3$ prime}
\end{array}\right. \\
2.{}^2B_2(8), \, 2.\mathbb{B} & \\ \hline\hline
\end{array}\]
\caption{The quasisimple groups $G$ with $Z(G) \ne 1$ and $\l(G)=4$}
\label{tab:quasi}
\end{table}
\renewcommand{\arraystretch}{1}

The next result sheds further light on the almost simple groups of depth $4$ arising in part (vi) of Theorem \ref{depth4}.

\begin{theorem}\label{asdepth4}
Suppose $G$ is an almost simple group of depth $4$. Then one of the following holds, where $T$ is the socle of $G$:
\begin{itemize}\addtolength{\itemsep}{0.2\baselineskip}
\item[{\rm (i)}] $G/T = C_p$, $p$ prime and $\l(T)=3$;
\item[{\rm (ii)}] $(G,T)$ is one of the cases in Table $\ref{tab:as1}$ (in each case, $\l(T)=4$);
\item[{\rm (iii)}] $G=T$ has a soluble maximal subgroup $M$ of chief length $3$, and 
$(G,M)$ is one of the cases in Table $\ref{tab:as2}$;
\item[{\rm (iv)}] $G=T$ has a simple maximal subgroup of depth $3$.
\end{itemize}
Moreover, all of the groups arising in cases (i), (ii) and (iii) have depth $4$.
\end{theorem}

In Table \ref{tab:as2}, the required conditions when $G = {\rm L}_{2}(q)$ and $q$ is odd are rather complicated to state (mainly due to the fact that we need $\l(G) \ne 3$). To simplify the presentation of the table, we refer to the following conditions on $q$ (recall that $\O(n)$ denotes the number of prime divisors of $n$, counting multiplicities):
\begin{equation}\label{cond}
\left\{\begin{array}{l}
\mbox{$\O(q \pm 1) \geqs 3$} \\
\mbox{$q \not\equiv \pm 3, \pm 13 \imod{40}$ if $q$ is prime} \\
\mbox{$q \ne 3^k$ with $k \geqs 3$ prime.}
\end{array}\right.
\end{equation}

We refer the reader to Section \ref{ss:as} for further details on the simple groups that arise in part (iv) of Theorem \ref{asdepth4}. It is worth noting that there exist simple groups of depth $3$ with a simple maximal subgroup of depth $3$. For instance, $A_5$ is a maximal subgroup of ${\rm L}_{2}(11)$, and both groups have depth $3$ (see Table \ref{tab:main}). Similarly, ${\rm L}_{3}(3) < \BM$ is another example.

\renewcommand{\arraystretch}{1.1}
\begin{table}[h]
\[\begin{array}{lll} \hline\hline
T & G & \mbox{Conditions} \\ \hline
A_6 & {\rm PGL}_{2}(9), \, {\rm M}_{10} & \\
A_7,\, A_{11}, \, A_{23} & S_7, \, S_{11}, \, S_{23} & \\
{\rm L}_{2}(q) & {\rm PGL}_{2}(q) & \mbox{$q$ prime, $q \equiv \pm 11, \pm 19 \imod{40}$, $\O(q \pm 1) \geqs 3$} \\
{\rm L}_{3}(4) & {\rm PGL}_{3}(4) & \\
{\rm U}_{3}(5) & {\rm PGU}_{3}(5) & \\ \hline\hline
\end{array}\]
\caption{The almost simple groups $G=T.p$ with $\l(G)=\l(T)=4$}
\label{tab:as1}
\end{table}
\renewcommand{\arraystretch}{1}

\renewcommand{\arraystretch}{1.1}
\begin{table}[h]
\[\begin{array}{lll} \hline\hline
G & M & \mbox{Conditions} \\ \hline
A_p & p{:}((p-1)/2) & \mbox{$p$ prime, $\O(p-1)=3$} \\
A_6 & S_4, \, 3^2{:}4 & \\
{\rm L}_{2}(q) & \mathbb{F}_q{:}((q-1)/2), \, D_{q-1} & \mbox{$q$ odd, $\O(q-1)=3$, \eqref{cond} holds}
\\
& D_{q+1} & \mbox{$q$ odd, $\O(q+1)=3$, \eqref{cond} holds} \\
& S_4 & \mbox{$q$ prime, $q \equiv \pm 1 \imod{8}$, \eqref{cond} holds} \\
& \mathbb{F}_q{:}(q-1), \, D_{2(q-1)} & \mbox{$q$ even, $\O(q-1)=2$, $\O(q+1) \geqs 2$} \\
 & D_{2(q+1)} & \mbox{$q$ even, $\O(q+1)=2$, $\O(q-1) \geqs 2$} \\
{\rm L}_{3}^{\e}(q) & (C_{q-\e})^2{:}S_3 & \mbox{$q \geqs 8$ even, $q-\e$ prime, $\O(q^2+\e q+1) \geqs 2$} \\
{\rm L}_{n}^{\e}(q) & \left(\frac{q^n-\e}{(q-\e)(n,q-\e)}\right)\!{:}n & \mbox{$n \geqs 3$ prime,
$\O(\frac{q^n-\e}{(q-\e)(n,q-\e)}) = 2$} \\
{}^2B_2(q) & D_{2(q-1)} & \O(q-1)=2 \\
& (q\pm \sqrt{2q}+1){:}4 & \mbox{$q\pm \sqrt{2q}+1$ prime, $\O(q-1) \geqs 2$} \\
{}^2G_2(q) & (q\pm \sqrt{3q}+1){:}6 & \mbox{$q\pm \sqrt{3q}+1$ prime, $q>3$} \\
{}^3D_4(q) & (q^4-q^2+1){:}4 & \mbox{$q^4-q^2+1$ prime} \\
{\rm J}_{1} & 7{:}6, \, 11{:}10, \, 19{:}6, \, 2^3{:}7{:}3 & \\
{\rm J}_{4} & 43{:}14 & \\
{\rm Ly} & 67{:}22 & \\
{\rm Fi}_{24}' & 29{:}14 & \\
{\rm Th} & 31{:}15 & \\ \hline \hline
\end{array}\]
\caption{The simple groups $G$ of depth $4$ with a soluble maximal subgroup $M$ of depth $3$}
\label{tab:as2}
\end{table}
\renewcommand{\arraystretch}{1}

\subsection{Main results on length}

Next we turn to our main results on the lengths of finite groups. Recall that the finite simple groups of small length were studied by Janko and Harada in the 1960s, beginning with \cite{Jan}, which classifies the simple groups of length $4$ (since $\l(G) \geqs 3$, Iwasawa's theorem implies that  $l(G) \geqs 4$ for every non-abelian simple group $G$). In a second paper \cite{Jan2}, Janko describes the simple groups of length $5$ and this was extended by Harada \cite{Harada} to length at most $7$. In both papers, the main results state that either $G = {\rm L}_{2}(q)$ for some unspecified prime powers $q$, or $G$ belongs to a short list of specific groups. Later work by Cameron, Solomon and Turull \cite{CST} gives the exact length of all alternating and sporadic groups, and several strong results on the lengths of simple groups of Lie type are presented in the series of papers 
\cite{SST, ST, ST2} from the early 1990s. We refer the reader to the start of Section \ref{sec:length} for further details.

Our next result extends the earlier work in \cite{Jan, Jan2, Harada} by giving a precise classification of the simple groups of length at most $9$. Of course, it should be noted that our proof relies on the Classification of Finite Simple Groups, which was not available to Janko or Harada.

\begin{theorem}\label{t:length}
Let $G$ be a non-abelian finite simple group. Then $l(G) \leqs 9$ if and only if $G$ is one of the groups recorded in Table $\ref{tab:length}$, where $p$ is a prime number.
\end{theorem}

\renewcommand{\arraystretch}{1.1}
\begin{table}[h]
\[\begin{array}{cll} \hline\hline
l(G) & G & \mbox{Conditions} \\ \hline
4 & A_5, \, {\rm L}_{2}(q) & \mbox{$q=p>5$, $\max\{\O(q \pm 1)\}=3$ and $q \equiv \pm 3, \pm 13 \imod{40}$} \\
5 & {\rm L}_{2}(q) & \mbox{$q \in \{ 7, 8, 9, 11, 19, 27, 29 \}$, or $q=p$ and $\max\{\Omega( q \pm 1) \} = 4$} \\
6 & A_7,  {\rm J}_1,  {\rm L}_{2}(q) & \mbox{$q \in \{25,125\}$, or $q=p$ and $\max\{\Omega( q \pm 1) \} = 5$} \\
7 & {\rm M}_{11}, {\rm U}_{3}(3),  {\rm U}_{3}(5) & \\
& {\rm L}_{2}(q) & \mbox{$q \in \{16, 32, 49, 121, 169\}$, or $q=p$ and $\max\{\O(q \pm 1)\} = 6$, or} \\
&  & \mbox{$q=p^3$, $\O(q-1)=4$ and $\O(q+1) \leqs 6$} \\
8 & {\rm M}_{12},  {}^2B_2(8), {\rm L}_{3}(3) & \\
& {\rm L}_{2}(q) & \mbox{$q=p$ and $\max\{\O(q \pm 1)\} = 7$, or} \\
& & \mbox{$q=p^2$, $\O(q-1)=6$ and $\O(q+1) \leqs 7$, or} \\
& & \mbox{$q=p^3$, $\O(q-1)=5$ and $\O(q+1) \leqs 7$, or} \\
& &  \mbox{$q=p^3$, $\O(q-1) \leqs 4$ and $\O(q+1) = 7$, or} \\
& & \mbox{$q=p^5$, $\O(q-1)=3$ and $\O(q+1) \leqs 7$} \\
9 & A_8, {\rm U}_{4}(2), {\rm L}_{3}(4) & \\
& {\rm U}_{3}(q) & \mbox{$q \in \{4,11,13,29\}$, or $q=p$, $\O(q \pm 1) = 3$, $\O(q^2-q+1) \leqs 8$, } \\
& & \mbox{$q \equiv 2 \imod{3}$ and $q \equiv \pm 3, \pm 13 \imod{40}$} \\
& {\rm L}_{2}(q) & \mbox{$q \in \{81,128,2187\}$, or $q=p$ and $\max\{\O(q \pm 1)\} = 8$, or} \\
& & \mbox{$q=p^2$, $\O(q-1) =7$ and $\O(q+1) \leqs 8$, or} \\
& & \mbox{$q=p^2$, $\O(q-1) =6$ and $\O(q+1) = 8$, or} \\
& & \mbox{$q=p^3$, $\O(q-1) = 6$ and $\O(q+1) \leqs 8$, or} \\
& & \mbox{$q=p^3$, $\O(q-1) \leqs 5$ and $\O(q+1) = 8$, or} \\
& & \mbox{$q=p^5$, $\O(q-1) = 4$ and $\O(q+1) \leqs 8$, or} \\
& & \mbox{$q=p^5$, $\O(q-1) = 3$ and $\O(q+1) = 8$} \\ \hline\hline
\end{array}\]
\caption{The simple groups $G$ of length at most $9$}
\label{tab:length}
\end{table}
\renewcommand{\arraystretch}{1}

The proof of Theorem \ref{t:length} is given in Section \ref{s:len}, together with the proof of the following corollary, which describes the structure of finite groups of small length. Recall that soluble groups $G$ have length $\O(|G|)$, so we focus on insoluble groups.

\begin{corol}\label{c:small}
Let $G$ be a finite insoluble group, in which case $l(G) \geqs 4$.
\begin{itemize}\addtolength{\itemsep}{0.2\baselineskip}
\item[{\rm (i)}] $l(G)=4$ if and only if $G$ is simple as in line $1$ of Table $\ref{tab:length}$.
\item[{\rm (ii)}] $l(G)=5$ if and only if one of the following holds:

\vspace{1mm}

\begin{itemize}\addtolength{\itemsep}{0.2\baselineskip}
\item[{\rm (a)}] $G$ is simple as in line $2$ of Table $\ref{tab:length}$; or
\item[{\rm (b)}] $G = T \times C_p$ with $T$ simple of length $4$ (as in Table $\ref{tab:length}$) and $p$ a prime; or
\item[{\rm (c)}] $G = {\rm SL}_{2}(q)$ or ${\rm PGL}_{2}(q)$, and either $q=5$, or $q>5$ is a prime such that $\max\{\O(q \pm 1)\}=3$ and $q \equiv \pm 3, \pm 13 \imod{40}$.
\end{itemize}
\item[{\rm (iii)}] $l(G)=6$ if and only if one of the following holds:

\vspace{1mm}

\begin{itemize}\addtolength{\itemsep}{0.2\baselineskip}
\item[{\rm (a)}] $G$ is simple as in line $3$ of Table $\ref{tab:length}$; or
\item[{\rm (b)}] $G = T \times C_p$, or a quasisimple group $p.T$, or an almost simple group $T.p$, where $T$ is simple of length $5$ (as in Table $\ref{tab:length}$) and $p$ a prime; the quasisimple groups occurring are ${\rm SL}_2(q)$, $3.{\rm L}_2(9)$, and the almost simple groups are ${\rm PGL}_2(q)$, ${\rm M}_{10}$, $S_6$, ${\rm L}_2(8).3$ and ${\rm L}_2(27).3$; or
\item[{\rm (c)}] $G={\rm L}_2(q) \times (p.r)$, $({\rm L}_2(q)\times p).2$, ${\rm SL}_2(q) \times p$ or $2.{\rm L}_2(q).2$, where $p,r$ are primes and ${\rm L}_2(q)$ has length $4$, as in Table $\ref{tab:length}$.
\end{itemize}
\end{itemize}
\end{corol}

Let $G = {\rm L}_{2}(q)$, where $q$ is a prime, and consider the conditions on $q$ in the first row of Table \ref{tab:length}. One checks that the first ten primes that satisfy the given conditions are as follows:
\[
q \in \{13,43,67,173,283,317,653,787,907,1867\},
\]
but it is not known if there are infinitely many such primes. The following more general problem is addressed in \cite{AST}: Does there exist an infinite set $\mathcal{S}$ of non-abelian finite simple groups and a positive integer $N$ such that $l(G) \leqs N$ for all $G \in \mathcal{S}$? The main result of \cite{AST} gives a positive answer to this question. The key ingredient is a purely number theoretic result \cite[Theorem C]{AST}, which states that for each positive integer $n$, there is an infinite set of primes $\mathcal{P}$ and a positive integer $N$ such that $\O(p^n-1) \leqs N$ for all $p \in \mathcal{P}$. More precisely, for $n=2$ they show that the conclusion holds with $N=21$, which immediately implies that there are infinitely many primes $p$ with $l({\rm L}_{2}(p)) \leqs 20$. The same problem arises in work of Gamburd and Pak (see \cite[p.416]{GP}), who state that $l({\rm L}_{2}(p)) \leqs 13$ for infinitely many primes $p$ (giving \cite{HR} as a reference).

We establish the following strengthening of the results in \cite{AST, GP}.

\begin{theorem}\label{t:inf}
There are infinitely many finite non-abelian simple groups $G$ with $l(G) \leqs 9$.
\end{theorem}

In fact we show that $l({\rm L}_{2}(p)) \leqs 9$ for infinitely many primes $p$.
As explained in Section \ref{s:hbb}, this is easily deduced from the following number-theoretic result of Heath-Brown, which is of independent interest.

\begin{theorem}[Heath-Brown]\label{th:hb}
There are infinitely many primes $p\equiv 5 \imod{72}$ for which
  \[\Omega((p^2-1)/24))\leqs 7.\]
\end{theorem}

See Appendix \ref{appendix} for the proof of this theorem, which implies that
there are infinitely many primes $p$ for which $\max\{\Omega(p\pm 1)\} \leqs 8$.

\subsection{Main results on chain differences and ratios}

Finally, we study the relationship between the length and depth of a finite group. Our first result determines the simple groups of chain difference two (see Section \ref{s:cd} for the proof). This extends earlier work of Brewster et al. \cite[Theorem 3.3]{BWZ} (also see Theorem \ref{t:bwz}), who described the simple groups of chain difference one.

\begin{theorem}\label{t:maincd}
Let $G$ be a finite simple group. Then ${\rm cd}(G)=2$ if and only if one of the following holds:
\begin{itemize}\addtolength{\itemsep}{0.2\baselineskip}
\item[{\rm (i)}] $G=A_7$, ${\rm J}_{1}$ or ${\rm U}_{3}(5)$.
\item[{\rm (ii)}] $G = {\rm L}_{2}(q)$ and either $q \in \{7,8,11,27,125\}$, or $q$ is a prime and one of the following holds:

\vspace{1mm}

\begin{itemize}\addtolength{\itemsep}{0.2\baselineskip}
\item[{\rm (a)}] $\max\{\O(q \pm 1)\} = 4$ and either $\min\{\O(q \pm 1)\} = 2$, or $q \equiv \pm 3, \pm 13 \imod{40}$.
\item[{\rm (b)}]$\max\{\O(q \pm 1)\} = 5$, $\min\{\O(q \pm 1)\} \geqs 3$ and $q \not\equiv \pm 3, \pm 13 \imod{40}$.
\end{itemize}
\end{itemize}
\end{theorem}

The \emph{chain ratio} of a finite group $G$ is given by ${\rm cr}(G) = l(G)/\l(G)$. By the aforementioned theorem of Iwasawa \cite{I}, ${\rm cr}(G)=1$ if and only if $G$ is supersoluble. Let us also observe that there are soluble, but not supersoluble, groups $G$ with the property that ${\rm cr}(G)$ is arbitrarily close to $1$. For example, if $G = S_4 \times C_{n}$, where $n$ is the product of the first $k$ primes, then $l(G) = k+4$ and $\l(G) \geqs k$. 

The next result establishes a best possible lower bound on the chain ratio of simple groups $G$. In particular, we see that ${\rm cr}(G)$ is bounded away from $1$, and Theorem \ref{t:cd} below shows that the same is true for all finite groups with trivial soluble radical.

\begin{theorem}\label{t:maincr}
Let $G$ be a non-abelian finite simple group. Then
\[
{\rm cr}(G) \geqs \frac{5}{4},
\]
with equality if and only if $l(G)=5$ and $\l(G)=4$.
\end{theorem}

It follows from \cite[Corollary 9]{BLS1} that there exists an absolute constant $a$
such that 
\[
l(G) \leqs a\, \cd(G)
\] 
for every non-abelian finite simple group $G$. As an immediate corollary of Theorem \ref{t:maincr}, we deduce that $a=5$ is the best possible constant.

\begin{corol}\label{c:cr}
Let $G$ be a non-abelian finite simple group. Then $l(G) \leqs 5\,{\rm cd}(G)$, with equality if and only if $l(G)=5$ and $\l(G)=4$.
\end{corol}

\begin{remark}\label{r:cr}
The simple groups $G$ with $l(G)=5$ and $\l(G)=4$ arising in Theorem \ref{t:maincr} and Corollary \ref{c:cr} can be determined by combining Theorem \ref{t:length} with \cite[Theorem 1]{BLS1}. The groups that arise are all of the form ${\rm L}_{2}(q)$ and either $q \in \{9,19,29\}$, or $q$ is a prime with $\max\{\O(q \pm 1)\}=4$,  $\min\{\O(q \pm 1)\} \geqs 3$ and $q \not\equiv \pm 3, \pm 13 \imod{40}$.
\end{remark}

Our final result, which applies Theorem \ref{t:maincr}, relates the structure of an arbitrary finite group $G$
with its chain difference. We let $R(G)$ denote the soluble radical of $G$.

\begin{theorem}\label{t:cd}
Let $G$ be a finite group. Then
\[
l(G/R(G)) \leqs 10\, \cd(G).
\]
In particular, if $R(G)=1$ then 
${\rm cr}(G) \geqs 10/9$.
\end{theorem}

Combining this theorem with \cite[Proposition 2.2]{AST}, it follows that
\[
\O(|G/R(G)|) \leqs 100\, \cd(G)^2.
\]

Note that the length of $G$ itself need not be bounded in terms of $\cd(G)$;
indeed, if $G$ is supersoluble then $\cd(G)=0$ while $l(G)$ may be arbitrarily large.
However, we show in Proposition \ref{p:ss} below, that, if $\ss(G)$ denotes the direct product of the non-abelian composition factors of $G$ (with multiplicities), then
\[
l(\ss(G)) \leqs 5\, \cd(G).
\]
This extends Corollary \ref{c:cr} dealing with simple groups, and serves as a useful tool
in the proof of Theorem \ref{t:cd} above.

The layout of the paper is as follows. After some preliminaries in Section \ref{s:prel}, we prove our main results on depth (Theorems \ref{depth3}--\ref{asdepth4}) in Section \ref{proofs:depth}. Section \ref{sec:length} contains the proofs of our main results on length, namely Theorems \ref{t:length} and \ref{t:inf}, and Corollary \ref{c:small}. Finally, in Section \ref{cdr} we consider chain differences and chain ratios, proving Theorems \ref{t:maincd}, \ref{t:maincr} and \ref{t:cd}.

\section{Preliminaries}\label{s:prel}

We begin by recording some preliminary results, which will be needed in the proofs of our main theorems. Given a finite group $G$, we write ${\rm chiefl}(G)$ for the length of a chief series of $G$. Recall that $l(G)$ and $\l(G)$ denote the length and depth of $G$, respectively, as defined in the Introduction. Let $\cd(G) = l(G)-\l(G)$ be the chain difference of $G$.

\begin{lem}\label{add}
Let $G$ be a finite group and let $N$ be a normal subgroup of $G$.
\begin{itemize}\addtolength{\itemsep}{0.2\baselineskip}
\item[{\rm (i)}] $l(G) = l(N)+l(G/N)$.
\item[{\rm (ii)}] $\l(G/N) \leqs \l(G) \leqs \l(N) + \l(G/N)$.
\item[{\rm (iii)}] If $N$ has prime order, then $\l(G) = \l(G/N)+1$.
\end{itemize}
\end{lem}

\begin{proof}
Part (i) is \cite[Lemma 2.1]{CST} for part (i), and part (ii) is very straighforward.

Now consider part (iii). Suppose $|N|=p$, a prime, let $t = \l(G)$, and let 
\begin{equation}\label{chai}
G = G_0 > G_1 > \cdots > G_t=1
\end{equation}
be an unrefinable chain of subgroups of $G$. Pick $i$ maximal such that $N\leqs G_i$. Then $G_{i+1} < G_{i+1}N \leqs G_i$, and so $G_{i+1}N = G_i$. Hence, writing $\bar G_i$ for the image of $G_i$ in $G/N$, we have $\bar G_{i+1} = \bar G_i$, and so taking images in the chain \eqref{chai} and deleting repetitions gives an unrefinable chain of length less than 
$t$ in $G/N$. Consequently $\l(G/N) < \l(G)$. Since also $\l(G/N) \geqs \l(G)-1$ by (ii), the conclusion follows. 
\end{proof}

Notice that Lemma \ref{add}(i) implies that the length of a finite group is equal to the sum of the lengths of its composition factors. In particular, if $G$ is soluble then $l(G) = \O(|G|)$, which is the number of prime divisors of $|G|$ (counting multiplicities).

The next result is \cite[Lemma 1.3]{BWZ}, which is an easy corollary of Lemma \ref{add}.

\begin{lem}\label{l:bwz}
If $G$ is a finite group, $B \leqs G$ and $A$ is a normal subgroup of $B$, then
\[
\cd(G) \geqs \cd(B/A) + \cd(A).
\]
In particular, ${\rm cd}(G) \geqs {\rm cd}(L)$ for every section $L$ of $G$.
\end{lem}

\begin{lem}\label{kolem}
Let $G$ be a finite group.
\begin{itemize}\addtolength{\itemsep}{0.2\baselineskip}
\item[{\rm (i)}] If $G$ is soluble, then $\l(G) = {\rm chiefl}(G)$.
\item[{\rm (ii)}] If $G$ is insoluble, then $\l(G) \geqs {\rm chiefl}(G)+2$.
\end{itemize}
\end{lem}

\begin{proof}
Part (i) is \cite[Theorem 2]{K} and part (ii) is \cite[Theorem 1.4]{SW}.
\end{proof}

\begin{lem}\label{fab}
Let $H$ be a finite nontrivial soluble group, $p$ a prime, and suppose $G = H^p\la \a\ra$ where $\a^p \in H^p$ and $\a$ permutes the $p$ factors transitively. Then $\l(G) \geqs \l(H)+2$.
\end{lem}

\begin{proof}
We proceed by induction on $\O(|H|)$, the number of prime factors of $|H|$ (counting multiplicities). For the base case, $H$ has prime order $q$. The chief length of $G = (C_q)^p.p$ is more than 2, so by Lemma \ref{kolem}(i), we have $\l(G) \geqs 3 = \l(H)+2$ in this case. 

Now assume that $|H|$ is not prime, and let $N$ be a minimal normal subgroup of $H$. Then $H/N \ne 1$, and by Lemma \ref{kolem}(i) we have $\l(H/N) = \l(H)-1$.
Let $M = \prod_{i=0}^{p-1} N^{\a^i}$. Then $1 \ne M \triangleleft G$, and $\l(G) \geqs \l(G/M)+1$, again by Lemma \ref{kolem}(i). Applying the induction hypothesis to $G/M \cong (H/N)^p.p$, we have
\[
\l(G/M) \geqs \l(H/N)+2.
\]
It follows that $\l(G) \geqs \l(H/N)+3 = \l(H)+2$, as required.
\end{proof}

The next lemma on the length of ${\rm L}_{2}(q)$ will be useful later.

\begin{lem}\label{l:easy}
Let $G = {\rm L}_{2}(q)$, where $q=p^f \geqs 5$ and $p$ is a prime.
\begin{itemize}\addtolength{\itemsep}{0.2\baselineskip}
\item[{\rm (i)}] If $q$ is even, then $l(G) = \Omega(q-1)+f+1$.
\item[{\rm (ii)}] If $q$ is odd, then either
\begin{equation}\label{e:le}
l(G) = \max\{\O(q-1)+f, \O(q+1)+1\},
\end{equation}
or $q \in \{7,11,19,29\}$ and $l(G)=5$, or $q=5$ and $l(G)=4$.
\end{itemize}
\end{lem}

\begin{proof}
Part (i) is a special case of \cite[Theorem 1]{ST}, noting that $2^f{:}(2^f-1)$ is a Borel subgroup of $G$. Now assume $q$ is odd. The case $f=1$ follows from \cite[Proposition 5.2]{CST}, so let us assume $f \geqs 2$. We proceed by induction on $\O(f)$.

First assume $\O(f)=1$, so $f$ is a prime, and let $M$ be a maximal subgroup of $G$. By inspecting
\cite[Tables 8.1, 8.2]{BHR}, either $M = {\rm PGL}_{2}(p)$ or $A_5$ (for $f=2$ only), or
$M = p^f{:}((p^f-1)/2)$, $D_{p^f \pm 1}$ or ${\rm L}_{2}(p)$, which gives
\[
l(G) = \max\{\O(q-1)+f, \O(q+1)+1, l({\rm L}_{2}(p))+1+\delta_{2,f}\},
\]
where $\delta_{i,j}$ is the familiar Kronecker delta. It is easy to check that \eqref{e:le} holds if $p \in \{3,5,7,11,19,29\}$. For example, if $p=29$ and $f=2$, then $\O(q-1)=6$ and $l({\rm L}_{2}(p)) = 5$. For any other prime $p$,
\[
l({\rm L}_{2}(p)) +1+\delta_{2,f} = \max\{\O(p\pm 1)\}+2+\delta_{2,f} \leqs \max\{\O(q-1)+f, \O(q+1)+1\}
\]
and the result follows.

Similarly, if $\O(f) \geqs 2$ then
\[
l(G) = \max\{\O(q-1)+f, \O(q+1)+1, l({\rm L}_{2}(q^{1/r}))+1+\delta_{2,r} \,:\, r \in \pi(f)\},
\]
where $\pi(f)$ is the set of prime divisors of $f$, and induction gives
\[
l({\rm L}_{2}(q^{1/r})) = \max\{\O(q^{1/r}-1)+f/r, \O(q^{1/r}+1)+1\}.
\]
Therefore
\[
l({\rm L}_{2}(q^{1/r}))+1+\delta_{2,r} \leqs \max\{\O(q-1)+f, \O(q+1)+1\}
\]
and we conclude that \eqref{e:le} holds.
\end{proof}

\section{Depth}\label{proofs:depth}

In this section we prove our results on the depth of finite groups, namely Theorems \ref{depth3}--\ref{asdepth4}.

\subsection{Proof of Theorem \ref{depth3}}

Let $G$ be a finite group. For soluble groups, the theorem is an immediate corollary of Lemma \ref{kolem}(i), so let us assume $G$ is insoluble. Here Lemma \ref{kolem}(ii) implies that $\l(G) \geqs {\rm chiefl}(G)+2$. Hence if $\l(G)=3$, then ${\rm chiefl}(G)=1$ and thus $G$ is simple, and is as in Table \ref{tab:main} by \cite[Theorem 1]{BLS1}. Conversely, the groups in the table indeed have depth $3$.

\subsection{Proof of Theorem \ref{depth4}}

Suppose $G$ is a finite group and $\l(G)=4$. If $G$ is soluble then it has chief length $4$ by Lemma \ref{kolem}(i), as in part (i) of Theorem \ref{depth4}. Now assume $G$ is insoluble. Since $\l(G) \geqs {\rm chiefl}(G)+2$ by Lemma \ref{kolem}(ii), it follows that ${\rm chiefl}(G) \leqs 2$. If ${\rm chiefl}(G)=1$ then $G$ is simple, so (vi) holds.

Now assume that ${\rm chiefl}(G)=2$. Then $G$ has a minimal normal subgroup $N \cong T^k$ for some simple (possibly abelian) group $T$, and $G/N \cong S$ is simple (also
possibly non-abelian).

Suppose first that $k=1$ and $S,T$ are both non-abelian. Then $G \cong S\times T$ by the Schreier hypothesis. When $S\not \cong T$, any maximal subgroup of $G$ is of the form $S_0\times T$ or $S \times T_0$ (with $S_0,T_0$ maximal in $S,T$ respectively), and neither of these can have depth 3, by Theorem \ref{depth3}. Hence $S\cong T$. Now $G$ has a maximal subgroup $M$ of depth 3, and $M$ cannot be of the above form $S_0\times T$ or $S \times T_0$. It follows that $M$ is a diagonal subgroup isomorphic to $T$, and hence $\l(T) = 3$ and $G$ is as in conclusion (ii) of Theorem \ref{depth4}.

Next suppose $k=1$ and $S$ or $T$ is an abelian simple group $C_p$. Then $G$ is one of $T\times C_p$, $S\times C_p$, a quasisimple group $p.S$ or an almost simple group $T.p$. The latter two possibilities are conclusions (v) and (vi). Now assume $G = T\times C_p$, and let $M$ be a maximal subgroup of $G$ of depth 3. Then $M$ is either $T$ or $T_0\times C_p$, where $T_0$ is maximal in $T$. In the latter case Theorem \ref{depth3} shows that $T_0\times C_p$ is soluble of chief length 3, hence ${\rm chiefl}(T_0)=2$. Therefore in both cases $T$ has depth 3, and so $G$ is as in conclusion (ii).

We may now assume that $k>1$.  Suppose $T$ is non-abelian and $S = C_p$. Then $k=p$ and $G= N\la \a\ra = T^p\la \a\ra$, where $\a^p \in N$ and $\a$ permutes the $p$ factors transitively.
Let $M$ be a maximal subgroup of $G$ of depth $3$. Then $M \ne N$ by Theorem \ref{depth3}, so, replacing $\a$ by another element in the coset $N\a$ if necessary, $M$ is of the form $\prod_{i=0}^{p-1}H^{\a^i}\la \a\ra \cong H^p.p$ for some maximal subgroup $H$ of $T$. Also $M$ is soluble, again by Theorem \ref{depth3}. But now Lemma \ref{fab} implies that $\l(M) \geqs \l(H)+2 \geqs 4$, which is a contradiction, so this case does not arise.

Next consider the case where $T = C_p$ and $S$ is non-abelian. Here $G= (C_p)^k.S$, where $S$ acts irreducibly on $V:=(C_p)^k$. Let $M$ be a maximal subgroup of $G$ of depth 3. If $V \not \leqs M$ then $M$ maps onto $S$, hence $M\cong S$ by Theorem \ref{depth3}, and conclusion (iii) holds. Now assume $V \leqs M$, so that $M$ is soluble, by Theorem \ref{depth3}. Then $M/V = S_0$, a maximal subgroup of $S$, and by Lemma \ref{kolem}(i), $\l(S_0) < \l(M)=3$. Hence $\l(S)=3$ and again (iii) holds.

It remains to handle the case where both $T$ and $S$ are non-abelian. Here $G = T^k.S$ and $S$ acts transitively on the $k$ factors. Let $M$ be a maximal subgroup of $G$ of depth 3. Then $T^k\not \leqs M$ by Theorem \ref{depth3}. Hence $M$ maps onto $S$, so $M\cong S$, again by Theorem \ref{depth3}. In particular, $\l(S)=3$.  Write $\O = (G:M)$, the coset space of $M$ in $G$. As $M$ is a core-free maximal subgroup of $G$, it follows that $G$ acts primitively on $\O$ and $T^k$ is a regular normal subgroup. At this point the O'Nan-Scott theorem (see \cite{LPS}, for example) implies that $G$ is a twisted wreath product $T\,{\rm twr}_\phi\,S$, as in conclusion (iv) of Theorem \ref{depth4}.

\vs

This completes the proof of Theorem \ref{depth4}.

\subsection{Examples for Theorem \ref{depth4}}\label{ss:exd4}

Consider the nonsplit groups $G = (C_p)^k.T$ arising in part (iii) of Theorem \ref{depth4}. Here $T$ is a simple group of depth $3$ (so the possibilities for $T$ are given in Table \ref{tab:main}) and $V = (C_p)^k$ is a nontrivial irreducible module for $T$ over $\mathbb{F}_q$ with $q=p^f$ for some $f \geqs 1$. In particular, $\dim V \geqs 2$, $p$ divides $|T|$ and the second cohomology $H^2(T,V)$ is nontrivial. As noted in Remark \ref{r:1}(b), $G$ has a maximal subgroup $M = (C_p)^k.S$ with $\l(M)=3$, where $S<T$ is maximal and acts irreducibly on $V$. Note that $S$ is soluble and has chief length $2$. It will be difficult to give a complete classification of the depth $4$ groups of this form, but we can identify some genuine examples:

\vspace{1mm}

\noindent \emph{Example.} Let $T = {\rm M}_{23}$, so $S = 23{:}11$. Now $T$ has an $11$-dimensional irreducible module $V$ over $\mathbb{F}_2$. Moreover, one checks that $S$ acts irreducibly on $V$ and $H^2(T,V) \ne 0$, hence there is a nonsplit group $2^{11}.{\rm M}_{23}$ of depth $4$.

\vspace{1mm}

\noindent \emph{Example.} Take $T = A_5$, $S = A_4$ and let $V$ be a $3$-dimensional irreducible module for $T$ over $\mathbb{F}_5$. Then $S$ acts irreducibly on $V$ and $H^2(T,V) \ne 0$, so there is a nonsplit group $5^{3}.A_5$ of depth $4$.

\vspace{1mm}

\noindent \emph{Example.} Suppose $T = {\rm L}_{n}(r)$, where $n \geqs 3$ is a prime and $(n,r-1)=1$, so
$S = \left(\frac{r^n-1}{r-1}\right){:}n$.
Let $V$ be the natural module for $T = {\rm SL}_{n}(r)$. Then $S$ acts irreducibly on $V$, and a theorem of Bell \cite{Bell} implies that $H^2(T,V) \ne 0$ if and only if
\[
(n,r) \in \{(3,3^a>3), (3,2), (3,5), (4,2), (5,2)\}.
\]
In particular, there is a nonsplit group $3^{9}.{\rm L}_{3}(27)$ of depth $4$ (note that we need $\l({\rm L}_{3}(r)) = 3$, which in this case means that $r^2+r+1$ is a prime). Thanks to Bell's result, there are also nonsplit groups $2^3.{\rm L}_{3}(2)$, $5^3.{\rm L}_{3}(5)$ and $2^5.{\rm L}_{5}(2)$, each of which has depth $4$.

\subsection{Proof of Theorem \ref{depth4q}}\label{ss:d4q}

Let $G$ be a quasisimple group with nontrivial centre and depth $4$. Write $G/Z(G)=T$, a non-abelian simple group. By Theorem \ref{depth4}, we have $Z(G)=C_p$ for a prime $p$, so Lemma \ref{add}(iii) implies that $\l(T)=3$. Theorem \ref{depth4q} now follows by considering the Schur multipliers of the simple groups in Table \ref{tab:main} (see \cite[Theorem 5.1.4]{KL}, for example).

\subsection{Proof of Theorem \ref{asdepth4}}

Let $G$ be an almost simple group with $\l(G)=4$ and socle $T$. First assume $G \ne T$, so Theorem \ref{depth4} implies that $G/T = C_p$ for a prime $p$. If $\l(T)=3$, then we are in case (i) of Theorem \ref{asdepth4}. Now assume $\l(T) \geqs 4$. We claim that (ii) holds, so $\l(T) = 4$ and $(G,T)$ is one of the cases in Table \ref{tab:as1}.

To see this, let $M$ be a maximal subgroup of $G$ of depth $3$. Then $M \ne T$, so $G=TM$ and $M \cap T \triangleleft M$ has index $p$. By Theorem \ref{depth3}, $M$ is soluble and Lemma \ref{kolem}(i) implies that $\l(M \cap T) = 2$, so $M \cap T$ is not maximal in $T$. Therefore, $G$ has a novelty soluble maximal subgroup $M$ of chief length $3$. This property is highly restrictive and we can determine all the possibilities for $G$ and $M$.

First assume $T = A_n$ is an alternating group, so $G = A_n.2$. If $n=6$ then $\l(T) = 4$ and one checks that $\l({\rm PGL}_{2}(9)) = \l({\rm M}_{10}) = 4$, while  $\l(S_6) = 5$. Now assume $n\ne 6$, so that $G = S_n$. By part (I) of the main theorem of \cite{LPS1}, the novelty soluble maximal subgroups of $S_n$ are $S_2\wr S_4 < S_8$ and $C_p{:}C_{p-1}<S_p$ for $p \in \{7,11,17,23\}$. Of these, only $C_p{:}C_{p-1}$ for $p \in \{7,11,23\}$ have depth $3$, so $S_7$, $S_{11}$ and $S_{23}$ are the only depth $4$ groups arising in this case.

If $T$ is a sporadic group then $G = T.2$ and one checks (by inspection of the Atlas \cite{Atlas}) that $G$ does not have a maximal subgroup $M$ with the required properties.

Now assume $T$ is a simple group of Lie type over $\mathbb{F}_q$. If $T$ is an exceptional group of Lie type, then all the maximal soluble subgroups of $G$ are known (see \cite{CLSS, LSS}) and one checks that there are no relevant examples (it is helpful to note that if $T = {}^2B_2(q)$, ${}^2G_2(q)$, ${}^2F_4(q)$ or ${}^3D_4(q)$, then $G$ does not have any novelty maximal subgroups). Finally, suppose $T$ is a classical group. For the low-rank groups, it is convenient to consult the relevant tables in \cite{BHR}; in this way, one checks that the only cases that arise are the ones listed in Table \ref{tab:as1} (in each case, $\l(T)=4$). For example, if $G = {\rm PGL}_{2}(q)$, where $q$ is a prime and $q \equiv \pm 11, \pm 19 \imod{40}$, then $G$ has a maximal subgroup $M = S_4$ and $M \cap T = A_4$ is non-maximal in $T$ (note that the additional condition $\O(q \pm 1) \geqs 3$ is needed to ensure that $\l(T) \geqs 4$, which means that $\l(T)=4$ by \cite[Lemma 3.1]{BLS1}). By inspecting \cite{KL}, it is easy to check that no examples arise when $G$ is one of the remaining classical groups not covered by \cite{BHR}. We conclude that part (ii) of Theorem \ref{asdepth4} holds.

To complete the proof, we may assume $G=T$ has depth $4$. Let $M$ be a maximal subgroup of $G$ with $\l(M)=3$. By Theorem \ref{depth3}, either $M$ is simple (and we are in part (iv) of Theorem \ref{asdepth4}), or $M$ is soluble of chief length $3$. It remains to show that in the latter case, the possibilities for $G$ and $M$ are given in Table \ref{tab:as2}. To do this, we essentially repeat the above argument, but now there are more cases to consider because $G=T$ and there is no novelty condition.

First assume $G = A_n$. It is easy to verify the result for $n \leqs 16$ (with the aid of {\sc Magma}, for example), so let us assume $n \geqs 17$. By the O'Nan-Scott theorem (see \cite{LPS}), the only soluble maximal subgroups of $G$ are of the form $M = {\rm AGL}_{1}(p) \cap G = C_p{:}C_{(p-1)/2}$, with $n = p$ a prime. Here Lemma \ref{kolem}(i) implies that $\l(M) = \O(p-1)$, which explains the condition $\O(p-1)=3$ in Table \ref{tab:as2}.

Next assume $G$ is a sporadic group. The groups with $\l(G) = 4$ can be read off from \cite[Lemma 3.3]{BLS1} and the cases appearing in Table \ref{tab:as2} are obtained by inspecting the lists of maximal subgroups of $G$ in the Atlas \cite{Atlas}.

Finally suppose $G$ is a simple group of Lie type over $\mathbb{F}_q$. As noted above, if $G$ is an exceptional group then all of the soluble maximal subgroups of $G$ are known and it is routine to read off the cases with such a subgroup of depth $3$ (for $G = {}^2B_2(q)$, note that we need the extra condition $\O(q-1) \geqs 2$ to ensure that $\l(G)=4$). Similarly, the result for classical groups is obtained by carefully inspecting \cite{BHR} (for the low-rank groups) and \cite{KL} (in the remaining cases). Once again, extra conditions on $q$ are needed to get $\l(G)=4$.

\vs

This completes the proof of Theorem \ref{asdepth4}.

\subsection{Examples for Theorem \ref{asdepth4}}\label{ss:as}

It is not feasible to give a complete description of the simple groups $G$ of depth $4$ with a simple maximal subgroup $M$ of depth $3$, as in part (iv) of Theorem \ref{asdepth4}, but we can give some partial information.

\vspace{1mm}

\noindent \emph{Sporadic groups.} Let $G$ be a sporadic group and recall that $\l(G)$ is recorded in \cite[Table 2]{BLS1}. By inspecting the Atlas \cite{Atlas}, excluding the Monster group $\mathbb{M}$, it is easy to see that the pairs $(G,M)$ with $\l(G)=4$ and $M$ a simple maximal subgroup of depth $3$ are as follows: 
\[
\begin{array}{lllll}
({\rm M}_{11}, {\rm L}_{2}(11)) & ({\rm M}_{12}, {\rm L}_{2}(11)) & ({\rm M}_{22}, {\rm L}_{2}(11)) & ({\rm M}_{24}, {\rm L}_{2}(7)) & ({\rm M}_{24}, {\rm L}_{2}(23)) \\
({\rm M}_{24}, {\rm M}_{23}) & ({\rm J}_{1}, {\rm L}_{2}(11)) & ({\rm J}_{2}, A_5) & ({\rm Suz}, {\rm L}_{2}(25)) & ({\rm Co}_2, {\rm M}_{23}) \\
({\rm Co}_3, {\rm M}_{23}) & ({\rm Fi}_{23}, {\rm L}_{2}(23)) & ({\rm Th},{\rm L}_{3}(3)) & &
\end{array}
\]
The pair $(\mathbb{M},{\rm L}_{2}(59))$ is another example (in particular, the Monster has depth $4$), but a complete list of the simple maximal subgroups of $\mathbb{M}$ of depth $3$ is not available.

\vspace{1mm}

\noindent \emph{Alternating groups.} Let $G = A_n$ be an alternating group. With the aid of {\sc Magma}, it is easy to check that for $n \leqs 100$, the possibilities for $(M,n)$ are as follows:
\[
\begin{array}{llllll}
(A_5,6) & ({\rm L}_{2}(7),7) & ({\rm L}_{3}(3),13) & ({\rm L}_{2}(13),14) & ({\rm M}_{23},23) & ({\rm L}_{3}(5),31) \\
({\rm L}_{5}(2),31) & ({\rm L}_{2}(37),38) & ({\rm L}_{2}(43),44) & ({\rm L}_{2}(47),48) & (A_{47},48) & ({\rm L}_{2}(53),54) \\
({\rm L}_{2}(59),60) & (A_{59},60) & ({\rm L}_{2}(61),62) & ({\rm L}_{2}(25),65) & ({\rm L}_{2}(67),68) & ({\rm L}_{2}(73),74) \\
({\rm L}_{2}(13),78) & ({\rm L}_{2}(83),84) & (A_{83},84) & (A_{87},88) & &
\end{array}
\]
The main theorem of \cite{LPS} on the maximal subgroups of symmetric and alternating groups provides some useful information in the general case, but it is not possible to state a precise result.

\vspace{1mm}

\noindent \emph{Exceptional groups.} Let $G$ be an exceptional group of Lie type over $\mathbb{F}_q$. If $G = {}^2B_2(q)$, ${}^2G_2(q)$, ${}^2F_4(q)'$, ${}^3D_4(q)$ or
$G_2(q)$, then the maximal subgroups of $G$ are known and one can read off the relevant examples $M$ with $M$ simple of depth 3: either $(G,M)$ is one of 
$({}^2F_4(2)', {\rm L}_{2}(25))$, $(G_2(3), {\rm L}_{2}(13))$, $(G_2(4), {\rm L}_{2}(13))$,
or 
\begin{itemize}\addtolength{\itemsep}{0.2\baselineskip}
\item $G = {}^2B_2(q)$ and $M={}^2B_2(q_0)$ with $q=q_0^k$, $q_0>2$ and both $k$ and $q_0-1$ are primes; or
\item $G = G_2(q)$ with $q=p^f$ for a prime $p \geqs 5$, $M = {\rm L}_{2}(13)$ or ${\rm L}_{2}(8)$, and the precise conditions on $q$ for the maximality of $M$ are given in  \cite[Table 8.41]{BHR}.
\end{itemize}
In the remaining cases, by combining results of Liebeck and Seitz \cite{LS99} with recent work of Craven \cite{Craven}, we deduce that there are no examples with $M$ an alternating or sporadic group. Strong restrictions on the remaining possibilities when $M$ is a group of Lie type can be obtained by applying \cite[Theorem 8]{LS_Durham} (defining characteristic) and the main theorem of \cite{LS99} (non-defining characteristic).

\vspace{1mm}

\noindent \emph{Classical groups.} Finally, suppose $G$ is a simple classical group with natural $n$-dimensional module $V$. By Aschbacher's subgroup structure theorem \cite{asch}, either $M$ belongs to a collection $\C(G)$ of \emph{geometric} subgroups, or $M \in \mathcal{S}(G)$ is almost simple and acts irreducibly on $V$. By inspecting \cite{BHR,KL}, it is possible to determine the relevant examples with $M \in \C(G)$ simple of depth 3 (the precise list of cases will depend on some delicate number-theoretic conditions). For example, suppose $G = {\rm U}_n(q)$ is a unitary group, where $n=q+1$ and $q \geqs 5$ is a prime. If $(q^{n-1}+1)/(q+1)$ is also a prime, then $G$ has a simple maximal subgroup $M = {\rm U}_{n-1}(q)$ of depth $3$ (here $M$ is the stabiliser of a non-degenerate $1$-space). For instance, $G = {\rm U}_{6}(5)$ has a maximal subgroup $M = {\rm U}_{5}(5)$ of depth $3$. It is not feasible to determine all the cases that arise with $M \in \mathcal{S}(G)$, although this can be achieved for the low-dimensional classical groups (that is, the groups with $n \leqs 12$) by inspecting the relevant tables in \cite[Chapter 8]{BHR}.

\section{Length} \label{sec:length}

In this section we prove our main results on length, namely Theorems \ref{t:length} and \ref{t:inf}, and Corollary \ref{c:small}. We begin by recalling some of the main results from the literature on the lengths of simple groups. 

The length of each alternating group is given by \cite[Theorem 1]{CST}, which states that  
\begin{equation}\label{e:an}
l(A_n) = \left\lfloor \frac{3n-1}{2}\right\rfloor - b_n - 1,
\end{equation}
where $b_n$ is the number of ones in the base $2$ expansion of $n$. Similarly, the length of each sporadic simple group is presented in \cite[Tables III and IV]{CST} (given more recent advances in our understanding of the maximal subgroups of sporadic groups, it is easy to verify that the ``Probable Values" recorded in \cite[Table IV]{CST} are correct). 

Now let $G$ be a finite simple group of Lie type over $\mathbb{F}_q$, where $q=p^f$ for a prime $p$. Let $r$ be the twisted Lie rank of $G$ and let $B$ be a Borel subgroup. By considering a descending chain of subgroups passing through $B$, it follows that 
\[
l(G) \geqs l(B)+r = \O(|B|)+r
\]
noting that $B$ is soluble. More precisely, if $p=2$ then \cite[Theorem 1]{ST} gives
\[
l(G) = l(B)+r+\e,
\]
where $\e=1$ if $G \cong {\rm U}_{2r+1}(2)$, otherwise $\e=0$. By \cite[Theorem A*]{ST2}, the same conclusion holds if $p>2$ and $q$ is sufficiently large.

Turning to Theorem \ref{t:length}, let $G$ be a non-abelian finite simple group. First recall that $\l(G) \geqs 3$ and ${\rm cd}(G) \geqs 1$, so $l(G) \geqs 4$.  The simple groups of length $4$ were classified by Janko \cite[Theorem 1]{Jan}. In a second paper \cite{Jan2}, he proved that every simple group of length $5$ is of the form ${\rm L}_{2}(q)$ for some prime power $q$ (but he did not give any further information on the prime powers that arise). In later work of Harada \cite{Harada}, this result was extended to simple groups of length at most $7$.

\begin{thm}[Harada, \cite{Harada}]\label{t:har}
Let $G$ be a finite simple group with $l(G) \leqs 7$. Then either
\begin{itemize}\addtolength{\itemsep}{0.2\baselineskip}
\item[{\rm (i)}] $G = {\rm U}_3(3), \, {\rm U}_3(5), \, A_7, \, {\rm M}_{11}, \, {\rm J}_1$; or
\item[{\rm (ii)}] $G = {\rm L}_{2}(q)$ for some prime power $q$.
\end{itemize}
\end{thm}

For the groups in part (i) of Theorem \ref{t:har}, it is easy to check that $A_7$ and ${\rm J}_1$ have length $6$, the others have length $7$. 

\subsection{Proof of Theorem \ref{t:length}}\label{s:len}

We are now ready to prove Theorem \ref{t:length}. In Lemma \ref{l:psl2} we first deal with the groups ${\rm L}_{2}(q)$ in all lengths, and then we handle the remaining simple groups, considering lengths $8$ and $9$ separately (see Lemmas \ref{l:l8} and \ref{l:l9}).

\begin{lem}\label{l:psl2}
Theorem \ref{t:length} holds if $G \cong {\rm L}_{2}(q)$.
\end{lem}

\begin{proof}
Write $q=p^f$, where $p$ is a prime. In view of Lemma \ref{l:easy}, the result is clear if $p=2$ or $f=1$, so let us assume $p \geqs 3$ and $f \geqs 2$, in which case
\[
l(G) = \max\{\O(q-1)+f, \O(q+1)+1\}.
\]
Since $l(G) \leqs 9$, it follows that $f \leqs 7$ and it is easy to verify the result when $p \in \{3,5\}$. Now assume $p \geqs 7$, in which case $\O(p^2-1) \geqs 5$ and $f \in \{2,3,5\}$.

Suppose $f=5$, so $\O(q-1) \geqs 3$ and we have $l(G)\in \{8,9\}$. For $l(G)=8$ we must have $\O(q-1)=3$ and $\O(q+1) \leqs 7$; one checks that there are primes $p$ with these properties:
\[
p \in \{3,7,23,83,263,1187, \ldots\}.
\]
Similarly, for $l(G)=9$ we need $\O(q-1)=4$ and $\O(q+1) \leqs 8$, or $\O(q-1)=3$ and $\O(q+1)=8$; in both cases, there are primes satisfying these conditions.

Next consider the case $f=3$, so $\O(q-1) \geqs 3$ and $l(G) \in \{6,7,8,9\}$. First assume $l(G)=6$, so $\O(q-1) = 3$ and thus $(p-1)/2$ and $p^2+p+1$ are both primes. In particular, $p \equiv -1 \imod{12}$ and $p>11$. Therefore, $\O(p+1) \geqs 4$ and $p^2-p+1$ is divisible by $3$, hence $\O(q+1) \geqs 6$ and we have reached a contradiction.
Next assume $l(G)=7$, so $\O(q-1)=3$ or $4$. If $\O(q-1)=3$ then we need $\O(q+1)=6$, which forces $\O(p+1)=4$ and $p^2-p+1=3r$ for some prime $r$. But $7$ divides $p^6-1$, so $7$ must divide $p+1$ and thus $p=83$ is the only possibility. But then  $p^2+p+1$ is composite, so this case does not arise. However, there are primes
\[
p \in \{7,11,83,1523, 20507, 28163, \ldots \}
\]
satisfying the conditions $\O(q-1)=4$ and $\O(q+1) \leqs 6$, so this case is recorded in Table \ref{tab:length}. Similarly, if $l(G)=m \in \{8,9\}$ then either $\O(q-1)=m-3$ and $\O(q+1) \leqs m-1$, or $\O(q-1) \leqs m-4$ and $\O(q+1)=m-1$. Moreover, one can check that there are primes $p$ satisfying these conditions.

Finally, let us assume $f=2$. Here the condition $p \geqs 7$ implies that $\O(q-1) \geqs 5$, so $l(G) \in \{7,8,9\}$. Suppose $l(G)=7$. Here $\O(q-1) = 5$ and either
$(p-1)/2$ or $(p+1)/2$ is a prime. Suppose $(p-1)/2$ is a prime, so $p \equiv 3 \imod{4}$ and $(p+1)/2 = 2r$ for some prime $r$. Therefore, $r$, $2r-1$ and $4r-1$ are all primes. If $r \in \{2,3\}$ then $p \in \{7,11\}$ and one checks that $l(G)=7$. Now assume $r \geqs 5$. If $r \equiv 1 \imod{3}$ then $4r-1$ is divisible by $3$. Similarly, if $r \equiv 2 \imod{3}$ then $3$ divides $2r-1$, so there are no examples with $r \geqs 5$. A similar argument applies if we assume $(p+1)/2$ is a prime: here we need a prime $r$ such that $2r+1$ and $4r+1$ are also primes, and one checks that $r=3$ is the only possibility, which corresponds to the case $G = {\rm L}_{2}(169)$ with $l(G)=7$.

Next assume $l(G)=8$ and $f=2$, so $\O(q-1)=5$ or $6$. The case $\O(q-1)=5$ is ruled out by arguing as in the previous paragraph. On the other hand, if $\O(q-1)=6$ then we need $\O(q+1) \leqs 7$ and there are primes $p$ with these properties. Finally, let us assume $l(G)=9$, so $\O(q-1) \in \{5,6,7\}$. The case $\O(q-1)=5$ is ruled out as above, whereas there are primes $p$ such that $\O(q-1)=7$ and $\O(q+1) \leqs 8$, or $\O(q-1)=6$ and $\O(q+1)=8$.
\end{proof}

In view of Theorem \ref{t:har}, it remains to determine the simple groups $G \not\cong {\rm L}_{2}(q)$ with $l(G)=8$ or $9$.

\begin{lem}\label{l:l8}
Theorem \ref{t:length} holds if $l(G)=8$.
\end{lem}

\begin{proof}
Let $G \not\cong {\rm L}_{2}(q)$ be a simple group with $l(G)=8$. If $G$ is a sporadic group, then by inspecting \cite[Tables III and IV]{CST} we see that $G = {\rm M}_{12}$ is the only example. From the formula in \eqref{e:an}, it is easy to check that no alternating group has length $8$.

Now assume $G$ is a simple group of Lie type over $\mathbb{F}_q$, where $q=p^f$ with $p$ a prime. First we handle the exceptional groups. If $G = {}^2B_2(q)$ then $p=2$, $f \geqs 3$ is odd and
\[
l(G) = \O(q-1)+2f+1
\]
by \cite[Theorem 1]{ST}, hence $l(G)=8$ if and only if $f=3$. Next assume $G = {}^2G_2(q)$, so $p=3$ and $f \geqs 3$ is odd. Since a Borel subgroup of $G$ has order $q^3(q-1)$, it follows that
\[
l(G) \geqs \O(q-1)+3f+1>8.
\]
If $G = G_2(q)$ then $q \geqs 3$ and $l(G) \geqs l({\rm GL}_{2}(q))+5f+1>8$, so no examples arise. All of the other exceptional groups can be eliminated in a similar fashion.

Finally, let us assume $G$ is a classical group. If $G = {\rm P\O}_{n}^{\e}(q)$ is an orthogonal group with $n \geqs 7$, then it is clear that $l(G)>8$ (indeed, a Sylow $p$-subgroup of $G$ has length greater than $8$). Similarly, we can eliminate symplectic groups ${\rm PSp}_{n}(q)$ with $n \geqs 6$. Now assume $G = {\rm PSp}_{4}(q)$ with $q=p^f \geqs 3$. Here $l(G) \geqs l({\rm SL}_{2}(q))+3f+1$, so we may assume $f=1$, in which case $G$ has a maximal subgroup $2^4.A_5$ or $2^4.S_5$ (according to the value of $q$ modulo $8$) and thus
$l(G) \geqs 1+4+l(A_5) = 9$. Similarly, it is easy to show that $l({\rm L}_{n}^{\e}(q)) > 8$ if $n \geqs 4$.

To complete the proof, we may assume that $G = {\rm L}_{3}^{\e}(q)$. Let $B$ be a Borel subgroup of $G$ and first assume $G = {\rm U}_{3}(q)$, so $q \geqs 3$. If $q$ is even, then \cite[Theorem 1]{ST} gives
\begin{equation}\label{e:u3}
l(G) = \O(|B|) + 1 = \O(q^2-1)+3f+1 - \O((3,q+1))
\end{equation}
and thus $l(G) \geqs 9$.
For $q$ odd we have
\begin{equation}\label{e:u32}
l(G) \geqs \O(q^2-1)+3f+1 - \O((3,q+1))
\end{equation}
and we quickly deduce that $f=1$, so $q \geqs 7$ (since ${\rm U}_{3}(3)$ and ${\rm U}_{3}(5)$ have length $7$). Now $\O(q^2-1) \geqs 5$, so we must have $\O(q^2-1)=5$ and $q \equiv 2 \imod{3}$, hence $(q-1)/2$ and $(q+1)/6$ are both prime. This implies that $q=6r-1$, where $r$ and $3r-1$ are primes, so $r=2$ is the only option and one checks that $l({\rm U}_{3}(11)) = 9$.

Finally, let us assume $G = {\rm L}_3(q)$. If $q$ is even then
\begin{equation}\label{e:l3}
l(G) = \O(|B|)+2 = 2\O(q-1)+3f+2-\O((3,q-1))
\end{equation}
and it is easy to see that $l(G) \ne 8$. Now assume $q$ is odd. If $q=3$ then one can check that $l(G)=8$, so let us assume $q \geqs 5$. Let $H$ be a maximal parabolic subgroup of $G$. Then $l(G) \geqs l(H)+1$, so
\begin{equation}\label{e:l32}
l(G) \geqs l({\rm L}_{2}(q)) +2f+2 +\O(q-1)- \O((3,q-1))
\end{equation}
and we deduce that $l(G) \geqs 9$.
\end{proof}

\begin{lem}\label{l:l9}
Theorem \ref{t:length} holds if $l(G)=9$.
\end{lem}

\begin{proof}
This is very similar to the proof of the previous lemma. Let $G \not\cong {\rm L}_{2}(q)$ be a simple group with $l(G)=9$. By inspection, $G$ is not a sporadic group. In view of \eqref{e:an}, $G = A_8$ is the only alternating group of length $9$. Now assume $G$ is a group of Lie type over $\mathbb{F}_q$, where $q=p^f$ with $p$ a prime. The exceptional groups are easily eliminated by arguing as in the proof of Lemma \ref{l:l8}. Similarly, if $G$ is a classical group then it is straightforward to reduce to the cases $G = {\rm PSp}_{4}(q)'$ and ${\rm L}_{3}^{\e}(q)$ (note  that ${\rm L}_{4}(2) \cong A_8$ and ${\rm U}_{4}(2) \cong {\rm PSp}_{4}(3)$).

Suppose $G = {\rm PSp}_{4}(q)'$. If $q=2$ then $G \cong A_6$ and $l(G) = 5$. Now assume $q \geqs 3$. As noted in the proof of the previous lemma, $l(G) \geqs l({\rm SL}_{2}(q))+3f+1$ and so we may assume $q=p$ is odd. Now $G$ has a maximal subgroup $H$ of type ${\rm Sp}_{2}(q) \wr S_2$, which implies that
\[
l(G) \geqs l(H)+1 = 3+2\,l({\rm L}_{2}(p)).
\]
If $p=3$ then this lower bound is equal to $9$ and one checks that $l({\rm PSp}_{4}(3)) = 9$. For $p>3$ we get $l(G) \geqs 11$.

Next assume $G = {\rm L}_{3}(q)$. If $q$ is even, then \eqref{e:l3} holds and one checks that $l(G)=9$ if and only if $q=4$. Now assume $q$ is odd. We have already noted that $l({\rm L}_{3}(3))=8$, so we may assume $q \geqs 5$. Moreover, in view of \eqref{e:l32}, we may assume that $q=p$. Since $l({\rm L}_{2}(p)) \geqs 4$ and $\O(p-1) \geqs 2$, it follows that $l({\rm L}_{2}(p)) = 4$, $\O(p-1) =2$ and $p \equiv 1 \imod{3}$. Clearly, $p=7$ is the only prime satisfying the latter two conditions, but $l({\rm L}_{2}(7))=5$.

To complete the proof of the lemma, we may assume $G = {\rm U}_{3}(q)$. If $q$ is even then \eqref{e:u3} holds and we deduce that $l(G)=9$ if and only if $q=4$. Now suppose $q$ is odd, so \eqref{e:u32} holds. If $f \geqs 2$ then $\O(q^2-1) \geqs 5$ and thus $l(G) \geqs 11$. Therefore, we may assume $q=p$ is odd. We have already noted that $l({\rm U}_{3}(3)) = l({\rm U}_{3}(5)) = 7$ and $l({\rm U}_{3}(11)) = 9$, and it is straightforward to check that $l({\rm U}_{3}(7))=10$. Now assume $q \geqs 13$. If $\O(q^2-1) \geqs 7$ then \eqref{e:u32} implies that $l(G) \geqs 10$, so we must have $\O(q^2-1) = 5$ or $6$.

If $\O(q^2-1)=5$ then $q=13$ is the only possibility (see the proof of Lemma \ref{l:psl2}) and one checks that $l({\rm U}_{3}(13))=9$. Now assume $\O(q^2-1)=6$, so $q \equiv 2 \imod{3}$ by \eqref{e:u32}. By considering the maximal subgroups of $G$ (see \cite[Tables 8.5 and 8.6]{BHR}), we see that
\[
l(G) = \max\{9,\O(q+1)+l({\rm L}_{2}(q))+1, \O(q^2-q+1)+1\}.
\]
Note that $\O(q+1) \geqs 3$ and $l({\rm L}_{2}(q)) \geqs 4$ since $q \geqs 13$ and $q \equiv 2 \imod{3}$. If $\O(q+1) \geqs 4$ then $l({\rm L}_{2}(q)) \geqs 5$ and thus $l(G) \geqs 10$. Therefore, $\O(q+1) = \O(q-1)= 3$ and $l({\rm L}_{2}(q)) \leqs 5$, so either $q=29$ or $q \equiv \pm 3, \pm 13 \imod{40}$. In addition, we need $\O(q^2-q+1) \leqs 8$ and one checks there are primes $p$ that satisfy these conditions:
\[
p \in \{173, 317, 653, 2693, 3413, 3677, \ldots\}. \qedhere
\]
\end{proof}

\vs

This completes the proof of Theorem \ref{t:length}.

\subsection{Proof of Corollary \ref{c:small}}\label{ss:csmall}

Here we prove Corollary \ref{c:small}, which describes the finite insoluble groups of length $4,5$ and $6$.

First observe that part (i) is clear from the additivity of length (see Lemma \ref{add}(i)) and the fact that every non-abelian simple group has length at least $4$.

Next, assume $G$ is a finite insoluble group with $l(G)=5$. The simple groups of length $5$ are given in Theorem \ref{t:length}, so we may assume that $G$ is not simple. Therefore, $G$ must have exactly two composition factors: a non-abelian simple group $T$ of length $4$ and depth $3$, and a cyclic group $C_p$ of prime order. In particular, $\l(G)=4$ and so $G$ is one of the groups in Theorem \ref{depth4}. By inspecting the various possibilities, we see that either $G = T \times C_p$, or $G$ is quasisimple with $G/Z(G)=T$ and $Z(G)=C_p$, or $G$ is almost simple with socle $T$ and $G/T = C_p$. Since $l(T) = 4$, \cite[Theorem 1]{Jan} implies that $T = {\rm L}_{2}(q)$ and $q$ is a prime satisfying the conditions in the first row of Table \ref{tab:length}. In particular, the only valid quasisimple and almost simple groups are of the form ${\rm SL}_{2}(q)$ and ${\rm PGL}_{2}(q)$, respectively. This completes the proof of part (ii) of Corollary \ref{c:small}.

Finally, suppose $G$ is insoluble of length 6. Again, the simple groups of length $6$ are given by Theorem \ref{t:length}, so we may assume $G$ is not simple. Then $G$ has a unique non-abelian composition factor $T$ of length 4 or 5,  and $6-l(T)$  abelian composition factors. It is readily checked that the possibilities for $G$ when $l(T) = 5$ (resp. 4) are those in (iii)(b) (resp. (c)) of Corollary \ref{c:small}.

\subsection{Proof of Theorem \ref{t:inf}}\label{s:hbb}

Set $G = {\rm L}_{2}(p)$, where $p \geqs 5$ is a prime. By Lemma \ref{l:easy}(ii),
\[
l(G) \leqs  1+ \max\{4,\O(p\pm 1)\}
\]
and Theorem \ref{t:hr} (see Appendix \ref{appendix}) implies that there are infinitely many primes $p$ such that $\max\{\O(p \pm 1)\}$ is at most $8$. The result follows.

\section{Chain differences and ratios}\label{cdr}

In this section we prove our main results concerning chain differences and chain ratios, namely Theorems \ref{t:maincd}, \ref{t:maincr} and \ref{t:cd}.

\subsection{Proof of Theorem \ref{t:maincd}}\label{s:cd}

Here we prove Theorem \ref{t:maincd}, which provides a classification of the simple groups with chain difference two. For comparison, we start by recalling \cite[Theorem 3.3]{BWZ}, which describes the simple groups of chain difference one.

\begin{thm}[Brewster et al. \cite{BWZ}]\label{t:bwz}
Let $G$ be a finite simple group. Then ${\rm cd}(G)=1$ if and only if $G = {\rm L}_{2}(q)$ and either $q \in \{4,5,9\}$, or $q$ is a prime and one of the following holds:
\begin{itemize}\addtolength{\itemsep}{0.2\baselineskip}
\item[{\rm (i)}] $3 \leqs \O(q\pm 1) \leqs 4$ and either $q \equiv \pm 1 \imod{10}$ or $q \equiv \pm 1 \imod{8}$.
\item[{\rm (ii)}] $\O(q\pm 1) \leqs 3$ and $q \equiv \pm 3, \pm 13 \imod{40}$.
\end{itemize}
\end{thm}

We begin the proof of Theorem \ref{t:maincd} by handling the alternating and sporadic groups in Lemma \ref{l:spal}. The simple groups of Lie type will be dealt with in Lemmas \ref{l:cdr2} and \ref{l:cdl2}, with the latter result covering the groups of the form ${\rm L}_{2}(q)$, which is the most difficult case. 

\begin{lem}\label{l:spal}
Let $G$ be a simple alternating or sporadic group. Then $\cd(G)=2$ if and only if $G = A_7$ or ${\rm J}_{1}$.
\end{lem}

\begin{proof}
First assume $G = A_n$ is an alternating group. A formula for $l(G)$ is given in \eqref{e:an} and it is easy to compute $\l(A_n)$ directly for small values of $n$: we get ${\rm cd}(A_5) = {\rm cd}(A_6)=1$, ${\rm cd}(A_7)=2$ and ${\rm cd}(A_8)=4$. By Lemma \ref{l:bwz}, it follows that
${\rm cd}(A_n) \geqs 4$ for all $n \geqs 8$.

Recall that the length and depth of each sporadic group $G$ is given in \cite[Tables III and IV]{CST} and \cite[Table 2]{BLS1}, respectively, and we immediately deduce that ${\rm cd}(G) = 2$ if and only if $G = {\rm J}_{1}$.
\end{proof}

\begin{lem}\label{l:cdr2}
Let $G$ be a simple group of Lie type over $\mathbb{F}_{q}$ with $G \not\cong {\rm L}_{2}(q)$. Then $\cd(G)=2$ if and only if $G = {\rm U}_{3}(5)$.
\end{lem}

\begin{proof}
Set $q=p^f$, where $p$ is a prime and $f \geqs 1$. We will follow a similar approach to the proof of \cite[Theorem 3.3]{BWZ} in the sense that we first handle the low-rank groups
\begin{equation}\label{e:lowr}
{\rm L}_{3}(q), \, {\rm U}_{3}(q), \, {\rm PSp}_{4}(q), \, {}^2F_4(q)',\, {}^2G_2(q), \, {}^2B_2(q)
\end{equation}
and we then appeal to Lemma \ref{l:bwz}. 

First assume $G = {\rm L}_{3}(q)$, in which case $q \geqs 3$ since ${\rm L}_{3}(2) \cong {\rm L}_{2}(7)$. If $q=3$ then $l(G)=8$ and $\l(G)=3$, so we may assume $q \geqs 4$ and thus
$l(G) \geqs 9$ by Theorem \ref{t:length}. If $p$ is odd then ${\rm L}_{3}(p)$ has a maximal subgroup ${\rm SO}_{3}(p) \cong {\rm PGL}_{2}(p)$, so $\l({\rm L}_{3}(p)) \leqs 6$ by \cite[Corollary 3.4]{BLS1} and thus ${\rm cd}({\rm L}_{3}(p)) \geqs 3$. In view of Lemma \ref{l:bwz}, this implies that ${\rm cd}(G) \geqs 3$ since ${\rm L}_{3}(p) \leqs G$. Now assume $p=2$ and let $H = QL$ be a maximal parabolic subgroup of $G$, where $Q$ is elementary abelian of order $q^2$ and $L \leqs {\rm GL}_{2}(q)$ has index $d=(3,q-1)$. Note that $L$ acts irreducibly on $Q$, so $L$ is a maximal subgroup of $H$. Therefore, $l(H) = \O((q-1)/d)+2f+l({\rm L}_{2}(q))$ and
\[
\l(H) \leqs \l(L)+1 \leqs \O((q-1)/d)+\l({\rm L}_{2}(q))+1,
\]
so ${\rm cd}(H) \geqs 2f-1+{\rm cd}({\rm L}_{2}(q)) \geqs 2f$ and the result follows since $f \geqs 2$.

Next assume $G = {\rm U}_{3}(q)$, so $q \geqs 3$. Let $H = QL$ be a Borel subgroup of $G$, where $d = (3,q+1)$. Here $Q = q^{1+2}$, $L = (q^2-1)/d$ and $Q/Z(Q)$ is elementary abelian of order $q^2$. Moreover, $L$ acts irreducibly on $Q/Z(Q)$. Therefore, $l(H) = 3f+\O(L)$ and
$\l(H) \leqs f+1+\O(L)$, so ${\rm cd}(H) \geqs 2f-1$ and we may assume $q=p$. If $p \in \{3,7,11\}$ then it is easy to check that ${\rm cd}(G) \geqs 3$, whereas ${\rm cd}(G) = 2$ if $p=5$. For $p>11$, Theorem \ref{t:length} implies that $l(G) \geqs 9$ and we get ${\rm cd}(G) \geqs 3$ since $G$ has a maximal subgroup ${\rm SO}_3(p) \cong {\rm PGL}_{2}(p)$ of depth at most $5$. 

To complete the analysis of the groups in \eqref{e:lowr}, we may assume 
\[
G \in \{ {\rm PSp}_{4}(q), {}^2F_4(q)', {}^2G_2(q), {}^2B_2(q)\}.
\] 
Suppose $G = {\rm PSp}_{4}(q)$ with $q \geqs 3$. If $q$ is even then $G$ has a maximal subgroup $H={\rm L}_{2}(q) \wr S_2$. Now $l(H) = 2\,l({\rm L}_{2}(q))+1$ and $\l(H) \leqs \l({\rm L}_{2}(q))+2$, so ${\rm cd}(H) \geqs l({\rm L}_{2}(q)) \geqs 4$. Similarly, if $q$ is odd then $H = 2^4.\O_{4}^{-}(2)<G$ and the result follows since ${\rm cd}(H)=4$. Next assume $G = {}^2F_4(q)'$. One checks that the
Tits group ${}^2F_4(2)'$ has depth $4$, so ${\rm cd}({}^2F_4(2)') \geqs 6$ by Theorem \ref{t:length} and thus ${\rm cd}(G) \geqs 6$ by Lemma \ref{l:bwz}. Now suppose $G = {}^2G_2(q)$, so $q=3^f$ and $f \geqs 3$ is odd. Let $H$ be a Borel subgroup of $G$ and let $K = 2 \times {\rm L}_{2}(q)$ be the centralizer in $G$ of an involution. Then
\begin{align*}
l(G) & \geqs l(H)+1 = \O(q-1)+3f+1 \\
\l(G) & \leqs \l(K)+1 \leqs \l({\rm L}_{2}(q))+2 \leqs \O(q-1)+3
\end{align*}
and thus ${\rm cd}(G) \geqs 3f-2 \geqs 7$. Finally, suppose $G = {}^2B_2(q)$, where $q=2^f$ and $f \geqs 3$ is odd. Here $l(G) = \O(q-1)+2f+1$ by \cite[Theorem 1]{ST} and $\l(G) \leqs \O(q-1)+2$ (since $G$ has a maximal subgroup $D_{2(q-1)}$). Therefore, ${\rm cd}(G) \geqs 2f-1 \geqs 5$.

We now complete the proof of the lemma by handling the remaining simple groups; the classical groups
\[
{\rm L}_{n}^{\e}(q) \, (n \geqs 4), \; {\rm PSp}_{n}(q) \, (n \geqs 6), \; {\rm P\O}_{n}^{\e}(q) \, (n \geqs 7)
\]
and the exceptional groups
\[
{}^3D_4(q), G_2(q), F_4(q), E_6^{\e}(q), E_7(q), E_8(q).
\]

Suppose $G = {\rm U}_{n}(q)$ with $n \geqs 4$. If $q$ is even then $G$ has a section isomorphic to ${\rm U}_{4}(2)$ and one checks that ${\rm cd}({\rm U}_{4}(2))=4$. Similarly, if $q$ is odd and $q \ne 5$ then $G$ has a section ${\rm U}_{3}(q)$ with ${\rm cd}({\rm U}_{3}(q)) \geqs 3$. Finally, suppose $q=5$. Since $\l({\rm U}_4(5)) = 5$ we get 
${\rm cd}({\rm U}_4(5)) \geqs 5$ by Theorem \ref{t:length}. The result follows since $G$ has a section isomorphic to ${\rm U}_{4}(5)$. 

In all of the remaining cases, it is easy to see that $G$ has a section isomorphic to ${\rm L}_{3}(q)$ and thus Lemma \ref{l:bwz} implies that ${\rm cd}(G) \geqs 3$ if $q \geqs 3$. Now assume $q=2$. If $G = G_2(2)' \cong {\rm U}_{3}(3)$ then ${\rm cd}(G)=3$. Since ${\rm U}_{3}(3)< {\rm Sp}_{6}(2) < \O_{8}^{-}(2)$, it follows that
${\rm cd}(G) \geqs 3$ if $G = {\rm Sp}_{6}(2)$ or $\O_{8}^{-}(2)$. In each of the remaining cases (with $q=2$), $G$ has a section isomorphic to ${\rm L}_{4}(2)$ and the result follows since ${\rm cd}({\rm L}_{4}(2))=4$.
\end{proof}

The next result completes the proof of Theorem \ref{t:maincd}.

\begin{lem}\label{l:cdl2}
If $G = {\rm L}_{2}(q)$ with $q \geqs 5$, then $\cd(G)=2$ if and only if
\begin{itemize}\addtolength{\itemsep}{0.2\baselineskip}
\item[{\rm (i)}] $q \in \{7,8,11,27,125\}$; or
\item[{\rm (ii)}] $q$ is a prime and one of the following holds:

\vspace{1mm}

\begin{itemize}\addtolength{\itemsep}{0.2\baselineskip}
\item[{\rm (a)}] $\max\{\O(q \pm 1)\} = 4$ and either $\min\{\O(q \pm 1)\} = 2$, or $q \equiv \pm 3, \pm 13 \imod{40}$.
\item[{\rm (b)}] $\max\{\O(q \pm 1)\} = 5$, $\min\{\O(q \pm 1)\} \geqs 3$ and $q \not\equiv \pm 3, \pm 13 \imod{40}$.
\end{itemize}
\end{itemize}
\end{lem}

\begin{proof}
As before, write $q=p^f$. First assume $p=2$, so $q \geqs 4$. Here $l(G) = \O(q-1)+f+1$ by Lemma \ref{l:easy}(i) and $\l(G) \leqs \O(q-1)+2$ (since $D_{2(q-1)}$ is a maximal subgroup). Therefore, ${\rm cd}(G) \geqs f-1$ and thus $f \in \{2,3\}$. If $f=2$ then ${\rm cd}(G)=1$, while ${\rm cd}(G) = 2$ if $f=3$ (the case $q=8$ is recorded in part (ii)(a) of Theorem \ref{t:maincd}).

Now assume $p \geqs 3$. For $q \leqs 11$, one checks that ${\rm cd}(G) = 2$ if and only if $q=7$ or $11$, so we may assume $q \geqs 13$. This implies that $D_{q-1}$ is a maximal subgroup of $G$ and thus $\l(G) \leqs \O(q-1)+1$. Now $l(G) \geqs \O(q-1)+f$ by Lemma \ref{l:easy}(ii), so ${\rm cd}(G) \geqs f-1$ and thus we may assume $f \in \{1,2,3\}$.

First assume $q=p \geqs 13$. By \cite[Corollary 3.4]{BLS1} we have
\[
\l(G) = \left\{\begin{array}{ll}
3 & \mbox{$\min\{\O(p\pm 1)\}=2$ or $p \equiv \pm 3, \pm 13 \imod{40}$} \\
4 & \mbox{otherwise.}
\end{array}\right.
\]
If $\l(G)=3$ then we need $l(G)=5$, in which case Theorem \ref{t:length} implies that $\max\{\O(p \pm 1)\}=4$. There are primes $p$ that satisfy these conditions. For example, if
\[
p \in \{23,59, 83, 227, 347, 563, \ldots\},
\]
then $\O(p-1) = 2$ and $\O(p+1)=4$. Similarly, we have $\l(G)=4$ and $l(G) = 6$ if and only if $\O(p \pm 1) \geqs 3$, $p \not\equiv \pm 3, \pm 13 \imod{40}$ and $\max\{\O(p \pm 1)\}=5$. Once again, there are primes $p$ with these properties.

Next assume $q=p^2$ with $p \geqs 5$. Since ${\rm PGL}_{2}(p)$ is a maximal subgroup of $G$, it follows that $\l(G) \leqs 6$ and thus $l(G) \leqs 8$. Now, if $l(G) \leqs 7$ then Theorem \ref{t:length} implies that $p \in \{5,7,11,13\}$ and in each case one checks that ${\rm cd}(G) \geqs 3$. Therefore, we may assume $\l(G)=6$ and $l(G)=8$. By Theorem \ref{t:length}, $l(G)=8$ if and only if $\O(q-1)=6$ and $\O(q+1) \leqs 7$. Similarly, $\l(G)=6$ if and only if $\O(q \pm 1) \geqs 5$, $p \equiv \pm 1 \imod{10}$ and $\l({\rm L}_{2}(p)) = 4$. The latter constraint yields the additional condition
$\O(p \pm 1) \geqs 3$, so we need $\O(p-1) = \O(p+1)=3$ since $\O(q-1)=6$. We claim that there are no primes $p$ that satisfy these conditions. For example, suppose $p \equiv 1 \imod{10}$. Then $(p-1)/10$ must be a prime. If $p \equiv 1 \imod{3}$ then $p=31$ is the only possibility, but this gives $\O(p+1)=5$. On the other hand, if $p \equiv 2 \imod{3}$ then $(p+1)/6$ is a prime. But $p^2 \equiv 1 \imod{8}$ and thus $(p-1)/10=2$ or $(p+1)/6=2$, which implies that $p=11$ and $\O(p-1)=2$. A very similar argument handles the case $p \equiv -1 \imod{10}$.

Finally, suppose $q=p^3$. If $p=3$ then one checks that $\l(G)=3$ and $l(G)=5$, so ${\rm cd}(G)=2$ in this case. Now assume $p \geqs 5$. Since ${\rm L}_{2}(p)$ is a maximal subgroup of $G$, it follows that
\[
\l(G) = \left\{\begin{array}{ll}
4 & \mbox{$\min\{\O(p\pm 1)\}=2$ or $p \equiv \pm 3, \pm 13 \imod{40}$} \\
5 & \mbox{otherwise}
\end{array}\right.
\]
and thus $l(G) \leqs 7$. By applying Theorem \ref{t:length}, we see that $\l(G)=4$ and $l(G)=6$ if and only if $p=5$. Similarly, $\l(G)=5$ and $l(G)=7$ if and only if $\O(q-1)=4$, $\O(q+1) \leqs 6$, $\O(p \pm 1) \geqs 3$ and $p \not\equiv \pm 3, \pm 13 \imod{40}$. Note that the conditions $\O(q-1)=4$ and $\O(p \pm 1) \geqs 3$ imply that $\O(p-1)=3$ and $p^2+p+1$ is a prime, so $p \equiv 2 \imod{3}$ and $p^2-p+1$ is divisible by $3$, whence $\O(p^2-p+1) \geqs 2$. There are primes $p$ such that $\O(p^3-1)=4$, $\O(p^3+1) \leqs 6$ and $\O(p \pm 1) \geqs 3$: the smallest one is $433373$. However, we claim that there is no prime $p$ that also satisfies the condition $p \not\equiv \pm 3, \pm 13 \imod{40}$.

If $p \equiv 1, 9, 17, 33 \imod{40}$ then $p-1$ is divisible by $8$ and thus $\O(p-1) \geqs 4$. Similarly, if $p \equiv 7,23,31,39 \imod{40}$ then $p+1$ is divisible by $24$, and $p \ne 23$ since we need $\O(p-1)=3$, so $\O(p+1) \geqs 5$. But we have already noted that $\O(p^2-p+1) \geqs 2$, whence $\O(p^3+1) \geqs 7$. Finally, suppose $p \equiv 11,19,29 \imod{40}$. These cases are similar, so let us assume $p \equiv 11 \imod{40}$. Here $(p-1)/10$ is a prime and $p+1$ is divisible by $12$, so $\O(p+1) \geqs 4$ since $p \ne 11$. Since $\O(p^3+1) \leqs 6$, it follows that $(p+1)/12$ and $(p^2-p+1)/3$ are both primes. Now $p^6 \equiv 1 \imod{7}$, so one of $(p-1)/10$, $p^2+p+1$, $(p+1)/12$ or $(p^2-p+1)/3$ must be equal to $7$, but it is easy to see that this is not possible. For example, if $(p-1)/10=7$ then $p = 71$ does not satisfy the required congruence condition.
\end{proof}

\subsection{Proof of Theorem \ref{t:maincr}}

Recall that ${\rm cr}(G) = l(G)/\l(G)$ is the \emph{chain ratio} of $G$. In this section we prove Theorem \ref{t:maincr}, which states that 
\[
{\rm cr}(G) \geqs \frac{5}{4}
\]
for every finite non-abelian simple group $G$, with equality if and only if $l(G) = 5$ and $\l(G)=4$ (all such groups are of the form ${\rm L}_{2}(q)$; see Remark \ref{r:cr} for further details). 

We partition the proof into several cases. First, Lemma \ref{l:salt} handles the sporadic and alternating groups, and Lemma \ref{l:c2} deals with the groups of the form ${\rm L}_{2}(q)$. The proof for the remaining groups of Lie type is covered by Lemmas \ref{l:ex} and \ref{l:c1}, where the exceptional and classical groups are handled, respectively.

\begin{lem}\label{l:salt}
Theorem \ref{t:maincr} holds if $G$ is a sporadic or alternating group.
\end{lem}

\begin{proof}
First assume $G$ is a sporadic group. The length and depth of $G$ is given in \cite[Tables III and IV]{CST} and \cite[Table 2]{BLS1}, respectively, and we immediately deduce that ${\rm cr}(G) \geqs 3/2$, with equality if and only if $G = {\rm J}_{1}$.

Now assume $G = A_n$ is an alternating group. The length of $G$ is given in \eqref{e:an} and
\cite[Theorem 2]{BLS1} states that $\l(G) \leqs 23$. One checks that this bound is sufficient if $n \geqs 23$. For example, if $n=23$ then $l(G) = 34-4-1=29$ and thus ${\rm cr}(G) \geqs 29/23>5/4$. For $n<23$ we can use
{\sc Magma} to show that $\l(G) \leqs 6$, with equality if and only if $n=16$. In view of the above formula for $l(G)$, we deduce that ${\rm cr}(G)>5/4$ if $n \geqs 8$. For the smallest values of $n$, we get ${\rm cr}(A_5) = 4/3$, ${\rm cr}(A_6)=5/4$ and ${\rm cr}(A_7)=3/2$.
\end{proof}

\begin{lem}\label{l:c2}
Theorem \ref{t:maincr} holds if $G \cong {\rm L}_{2}(q)$.
\end{lem}

\begin{proof}
Write $q=p^f$ with $p$ a prime. If $f=1$ then $\l(G) \in \{3,4\}$ by \cite[Corollary 3.4]{BLS1}, and we have $l(G) \geqs \l(G)+1$, so ${\rm cr}(G) \geqs 5/4$ and equality holds if and only if $\l(G)=4$ and $l(G)=5$. The result now follows by combining Theorems \ref{t:length} and \ref{t:bwz}. For the remainder, we may assume $f \geqs 2$.

Suppose $p=2$. Here $l(G) = \O(q-1)+f+1$ and $\l(G) \leqs \O(q-1)+\O(f)+1$ by \cite[Theorem 1]{ST} and \cite[Theorem 4(i)]{BLS1}. Now
\[
\O(q-1)+f+1 > \frac{5}{4}(\O(q-1)+\O(f)+1)
\]
if and only if
\begin{equation}\label{e:now}
\O(q-1)+5\O(f)+1 < 4f.
\end{equation}
Since
\[
\O(q-1)+5\O(f)+1 < f+5\log_2f+1,
\]
it is routine to check that \eqref{e:now} holds for all $f \geqs 2$.

Now suppose $p>2$ and observe that $l(G) \geqs \O(q-1)+f$ by Lemma \ref{l:easy}(ii). First assume $f \geqs 3$ is odd. By considering a chain of subfield subgroups (as in the proof of \cite[Theorem 4]{BLS1}), we deduce that $\l(G) \leqs \O(f)+\l({\rm L}_{2}(p))$. Therefore, $\l(G) \leqs \O(f)+2$ if $p=3$, so
\[
l(G) \geqs \O(q-1)+f \geqs f+2 > \frac{5}{4}(\log_3f+2) \geqs \frac{5}{4}(\O(f)+2) \geqs \frac{5}{4}\l(G)
\]
as required. Similarly, if $p \geqs 5$ then $l(G) \geqs f+3$, $\l(G) \leqs \O(f)+4$ and for $f>3$ the result follows in the same way. If $f=3$ then $\l(G) \leqs 5$ and Theorem \ref{t:bwz} implies that ${\rm cd}(G) \geqs 2$, so ${\rm cr}(G)>5/4$.

Finally, let us assume $p>2$ and $f$ is even. If $q=9$ then $G \cong A_6$ and we have already noted that ${\rm cr}(G)=5/4$ in this case. Now assume $q>9$, so $\O(q-1) \geqs 4$ and thus $l(G) \geqs f+4$. Also observe that $\l(G) \leqs 2\O(f)+\l({\rm L}_{2}(p))$. If $p=3$ then $f \geqs 4$, $\l(G) \leqs 2\O(f)+2$ and one checks that
\[
f+4 > \frac{5}{4}(2\log_2f +2),
\]
which gives the desired result. Now assume $p \geqs 5$. Here $\l(G) \leqs 2\O(f)+4$ and we have
\[
f+4 > \frac{5}{4}(2\log_2f +4)
\]
if $f>8$. If $f \in \{4,6,8\}$ then $\O(q-1) \geqs 6$, so $l(G) \geqs f+6$ and the result follows.
Finally, if $f=2$ then $\l(G) \leqs 6$ (since ${\rm PGL}_{2}(p)<G$ is maximal) and ${\rm cd}(G) \geqs 2$ by Theorem \ref{t:bwz}, so ${\rm cr}(G)>5/4$ as required.
\end{proof}

\begin{lem}\label{l:ex}
Theorem \ref{t:maincr} holds if $G$ is an exceptional group of Lie type.
\end{lem}

\begin{proof}
Let $G$ be a finite simple exceptional group of Lie type over $\mathbb{F}_q$, where $q=p^f$ and $p$ is a prime. Let $B$ be a Borel subgroup of $G$ and let $r$ be the twisted Lie rank of $G$. Then
\begin{equation}\label{e:bd}
l(G) \geqs \O(|B|)+r
\end{equation}
and \cite[Theorem 4]{BLS1} gives
\[
\l(G) \leqs 3\O(f)+36
\]
if $G \ne {}^2B_2(q)$.

First assume $G = E_8(q)$. Here $|B|=q^{120}(q-1)^8$ so
\[
l(G) \geqs 120f+8 > \frac{5}{4}(3\log_2f +36) \geqs \frac{5}{4}(3\O(f)+36) \geqs \frac{5}{4}\l(G)
\]
and the result follows. The case $E_7(q)$ is handled in exactly the same way, and similarly $E_6^{\e}(q)$ and $F_4(q)$ with $f \geqs 2$. Suppose $G = F_4(p)$, so $l(G) \geqs 28$ by \eqref{e:bd}. If $p=2$ then
${}^2F_4(2)<G$ is maximal and $\l({}^2F_4(2)) =5$, so $\l(G) \leqs 6$ and the result follows.
Similarly, if $p$ is odd then
\[
\l(F_4(p)) \leqs \l(2.\O_9(p)) +1 \leqs \l(\O_9(p))+2
\]
and one of $A_{10}$, $S_{10}$ or $A_{11}$ is a maximal subgroup of $\O_9(p)$ (see \cite[Table 8.59]{BHR}). Therefore $\l(\O_9(p)) \leqs 7$, so $\l(G) \leqs 9$ and once again we deduce that ${\rm cr}(G) > 5/4$. If $G = E_6^{\e}(p)$ then $F_4(p)<G$ is maximal, so the previous argument yields $\l(G) \leqs 10$ and the result quickly follows.

Next assume $G=G_2(q)'$. If $q=2$ then $G \cong {\rm U}_{3}(3)$ and one checks that $\l(G) = 4$ and $l(G) = 7$. Now assume $q>2$. Since $|B|=q^6(q-1)^2$ it follows that $l(G) \geqs 6f+4$ and by considering a chain of subfield subgroups we deduce that $\l(G) \leqs \O(f)+\l(G_2(p))$. If $p \geqs 5$ then $G_2(2)<G_2(p)$ is maximal, so $\l(G_2(p)) \leqs 6$ and it is easy to check that the same bound holds if $p=2$ or $3$. Therefore $\l(G) \leqs \O(f)+6$ and we deduce that
\[
l(G) \geqs 6f+4 > \frac{5}{4}(\log_2f + 6) \geqs \frac{5}{4}(\O(f)+6) \geqs \frac{5}{4}\l(G).
\]
If $G = {}^3D_4(q)$ then $G_2(q)<G$ is maximal and thus $\l(G) \leqs \O(f)+7$. In addition, $|B|=q^{12}(q^3-1)(q-1)$, so $l(G) \geqs 12f+2$ and the result follows.

To complete the proof of the lemma, we may assume $G = {}^2F_4(q)'$, ${}^2G_2(q)$ or ${}^2B_2(q)$. Suppose  $G = {}^2F_4(q)'$, so $q=2^f$ with $f$ odd. If $f=1$ then $\l(G) = 4$ and $l(G) \geqs 13$ since $G$ has a soluble maximal subgroup of the form $2.[2^8].5.4$. Similarly, if $f>1$ then $\l(G) \leqs \O(f)+5$ (see the proof of \cite[Theorem 4]{BLS1}), $l(G) \geqs 12f+2$ and these bounds are sufficient. The case $G = {}^2G_2(q)'$, where $q=3^f$ with $f$ odd, is very similar. If $f=1$ then $G \cong {\rm L}_{2}(8)$, so $\l(G)=3$ and $l(G)=5$. If $f>1$, then the proof of \cite[Theorem 4]{BLS1} gives $\l(G) \leqs \O(f)+4$ and we have $l(G) \geqs 3f+2$ since $|B|=q^3(q-1)$. It is easy to check that these bounds are sufficient.

Finally, let us assume $G = {}^2B_2(q)$, where $q=2^f$ with $f \geqs 3$ odd. Since $|B|=q^2(q-1)$, it follows that $l(G) \geqs 2f+1+\O(q-1) \geqs 2f+2$. By \cite[Theorem 4]{BLS1}, we also have
\[
\l(G) \leqs \O(f)+1+\O(q-1) < \O(f)+f+1.
\]
Therefore,
\[
l(G) \geqs 2f+2 > \frac{5}{4}(\log_2f + f+1) \geqs \frac{5}{4}\l(G)
\]
as required.
\end{proof}

\begin{lem}\label{l:c1}
Theorem \ref{t:maincr} holds if $G$ is a classical group.
\end{lem}

\begin{proof}
Let $G$ be a finite simple classical group over $\mathbb{F}_q$ and $r$ be the twisted rank of $G$. As before, write $q=p^f$ with $p$ a prime. Let $B$ be a Borel subgroup of $G$ and recall that \eqref{e:bd} holds. Our initial aim is to reduce the problem to groups of small rank. To do this, we will consider each family of classical groups in turn. In view of Lemma \ref{l:c2}, we may assume that $G \not\cong {\rm L}_{2}(q)$.

First assume $G = {\rm L}_{r+1}(q)$. We claim that ${\rm cr}(G)>5/4$ if $r \geqs 9$. To see this, first observe that
\[
l(G) \geqs \O(|B|)+r \geqs \frac{1}{2}fr(r+1)+r
\]
and $\l(G) \leqs 3\O(f)+36$ by \cite[Theorem 4]{BLS1}. For $r \geqs 9$, it is routine to check that
\[
l(G) \geqs \frac{1}{2}fr(r+1)+r > \frac{5}{4}(3\log_2f+36) \geqs \frac{5}{4}\l(G),
\]
which justifies the claim. In a similar fashion, we can reduce the problem to $r \leqs 6$ when $G = {\rm PSp}_{2r}(q)$, $\O_{2r+1}(q)$ or ${\rm P\O}_{2r}^{+}(q)$; $r \leqs 5$ when $G= {\rm P\O}_{2r+2}^{-}(q)$; and $r \leqs 4$ for $G = {\rm U}_{2r}(q)$.

Finally, suppose $G = {\rm U}_{2r+1}(q)$. Here $l(G) \geqs fr(2r+1)+r$ and \cite[Theorem 4]{BLS1} states that $\l(G) \leqs 3\O(f)+36$ if $q$ or $f$ is odd. If $q=2^f$ and $f$ is even, then the same theorem gives
\[
\l(G) \leqs 3\O(f)+35+2\O(2^{2^a}+1),
\]
where $f=2^ab$ and $b$ is odd. Since $\O(2^{2^a}+1) \leqs f$, it follows that $\l(G) \leqs 3\O(f)+2f+35$ for all possible values of $q$ and $f$, and one checks that
\[
l(G) \geqs fr(2r+1)+r > \frac{5}{4}(3\log_2f+2f+35) \geqs \frac{5}{4}(3\O(f)+2f+35) \geqs \frac{5}{4}\l(G)
\]
if $r \geqs 5$.

Therefore, in order to complete the proof of the lemma, we may assume that we are in one of the following cases, which will be treated in alphabetical order: 
\begin{itemize}\addtolength{\itemsep}{0.2\baselineskip}
\item[{\rm (a)}] $G = \O_{2r+1}(q)$ with $3 \leqs r \leqs 6$ and $q$ odd;
\item[{\rm (b)}] $G = {\rm PSp}_{2r}(q)$ with $2 \leqs r \leqs 6$;
\item[{\rm (c)}] $G = {\rm P\O}_{2r}^{+}(q)$ with $4 \leqs r \leqs 6$;
\item[{\rm (d)}] $G = {\rm P\O}_{2r+2}^{-}(q)$ with $3 \leqs r \leqs 5$;
\item[{\rm (e)}] $G = {\rm L}_{r+1}(q)$ with $2 \leqs r \leqs 8$;
\item[{\rm (f)}] $G = {\rm U}_{2r}(q)$ with $2 \leqs r \leqs 4$;
\item[{\rm (g)}] $G = {\rm U}_{2r+1}(q)$ with $1 \leqs r \leqs 4$.
\end{itemize}

Let us start by handling case (a). By considering a chain of subfield subgroups, we see that $\l(G) \leqs 2\O(f)+\l(\O_{2r+1}(p))$. In addition, the proof of \cite[Theorem 4]{BLS1} implies that
$\l(\O_{2r+1}(p)) \leqs 4+\l(S_n)$ for some $n \leqs 2r+3 = 15$. One checks that $\l(S_n) \leqs 6$ for $n \leqs 15$, hence $\l(G) \leqs 2\O(f) +10$. Now $|B| = \frac{1}{2}(q-1)^rq^{r^2}$, so \eqref{e:bd} yields
$l(G) \geqs fr^2 + 2r-1$ and one checks that
\[
fr^2+2r-1 > \frac{5}{4}(2\log_2f+10)
\]
for all possible values of $f$ and $r$. The result follows.

Next consider (b). First assume $p=2$, in which case $\l(G) \leqs \O(f)+\l({\rm Sp}_{2r}(2))$ and
$\l({\rm Sp}_{2r}(2)) \leqs 4+\l(S_n)$ for some $n \leqs 14$ (see the proof of \cite[Theorem 4]{BLS1}). Therefore, $\l(G) \leqs \O(f)+10$. Since $l(G) \geqs fr^2+r$ by \eqref{e:bd}, the result follows unless $(r,f) = (3,1)$, or if $r=2$ and $f\leqs 3$. If $(r,f) = (2,1)$ then $G \cong A_6$ and ${\rm cr}(G)=5/4$. In each of the remaining cases we have $\l(G) \leqs 5$ and ${\rm cd}(G) \geqs 2$, which implies the desired bound. Now assume $p$ is odd, so $l(G) \geqs fr^2+2r-1$ and the proof of \cite[Theorem 4]{BLS1} yields
\[
\l(G) \leqs 2\O(f) + \l({\rm PSp}_{2r}(p)) \leqs 2\O(f)+8+\l(S_r)\leqs 2\O(f)+13.
\]
This gives ${\rm cr}(G)>5/4$ unless $(r,f) = (3,1)$, or if $r=2$ and $f\leqs 4$. If $G = {\rm PSp}_{6}(p)$ then
\[
G>{\rm L}_{2}(p^3).3 > {\rm L}_{2}(p^3) > {\rm L}_{2}(p)
\]
is unrefinable, so $\l(G) \leqs 7$ and the result follows since ${\rm cd}(G) \geqs 2$. Now assume $r=2$ and $f \leqs 4$. Note that $|B|=\frac{1}{2}q^4(q-1)^2$, so $\O(|B|) = 4f+2\O(q-1)-1$. If $f=4$ then $\l(G) \leqs 4+\l({\rm PSp}_{4}(p))$ and we note that one of $A_6$, $S_6$ or $S_7$ is a maximal subgroup of ${\rm PSp}_{4}(p)$, so $\l({\rm PSp}_{4}(p)) \leqs 7$ and thus $\l(G) \leqs 11$. In addition, $\O(q-1) \geqs 5$, so $l(G) \geqs 27$ and the result follows. Similarly, if $f=2$ or $3$ then $\l(G) \leqs 8$ and $l(G) \geqs 15$. Finally, if $f=1$ then $\l(G) \leqs 7$ and the result follows since ${\rm cd}(G) \geqs 2$.

Now let us turn to case (c), so $G = {\rm P\O}_{2r}^{+}(q)$ and $r = 4,5$ or $6$. First assume $p=2$, in which case $l(G) \geqs fr(r-1)+r$ and $\l(G) \leqs \O(f)+\l(\O_{2r}^{+}(2))$. It is easy to check that $\l(\O_{2r}^{+}(2)) \leqs 9$. For example, if $r=6$ then there is an unrefinable chain
\[
\O_{12}^{+}(2) > {\rm Sp}_{10}(2) > \O_{10}^{-}(2).2 > \O_{10}^{-}(2) > A_{12}
\]
and $\l(A_{12})=5$, so $\l(\O_{12}^{+}(2)) \leqs 9$. Therefore, $l(G) \geqs 12f+4$, $\l(G) \leqs \O(f)+9$ and one checks that these bounds are sufficient. Now assume $p>2$. Here $l(G) \geqs fr(r-1)+2r-2$ and $\l(G) \leqs 3\O(f)+\l({\rm P\O}_{2r}^{+}(p))$. One checks that $\l({\rm P\O}_{2r}^{+}(p)) \leqs 9$. For instance, if $r=6$ then there is an unrefinable chain
\[
{\rm P\O}_{12}^{+}(p) > {\rm PSO}_{11}(p) > \O_{11}(p) > H
\]
with $H = A_{12}, S_{12}$ or $A_{13}$, and the claim follows since $\l(H) \leqs 6$. Therefore, $l(G) \geqs 12f+6$, $\l(G) \leqs 3\O(f)+9$ and we conclude that ${\rm cr}(G)>5/4$. A very similar argument applies in case (d) and we omit the details.

Next consider case (e), so $G = {\rm L}_{r+1}(q)$ and $\l(G) \leqs 2\O(f)+\l({\rm L}_{r+1}(p))$. Note that
\[
|B| = \frac{q^{r(r+1)/2}(q-1)^{r}}{(r+1,q-1)}.
\]
Suppose $r \in \{3,5,7\}$ is odd. Now ${\rm L}_{r+1}(p)$ has a maximal subgroup of the form ${\rm PSp}_{r+1}(p)$ or ${\rm PSp}_{r+1}(p).2$, and we noted that $\l({\rm PSp}_{r+1}(p)) \leqs 13$ in the analysis of case (b), whence $\l(G) \leqs 2\O(f)+15$. One now checks that the bound $l(G) \geqs fr(r+1)/2+r$ from \eqref{e:bd} is sufficient when $r=5$ or $7$. Now suppose $r=3$. If $q=2$ then ${\rm cr}(G)=9/5$ so we can assume $q>2$, in which case $l(G) \geqs 6f+5$. Now $\l({\rm L}_{4}(p)) = 5$ if $p=2$ or $3$, and $\l({\rm L}_{4}(p)) \leqs 9$ if $p \geqs 5$ (this follows from the fact that ${\rm PSp}_{4}(p).2$ is a maximal subgroup of ${\rm L}_{4}(p)$). Therefore, $\l(G) \leqs 2\O(f)+9$ and the result follows if $f>1$. Finally suppose $G = {\rm L}_{4}(p)$ with $p \geqs 3$. If $p=3$ then $\l(G)=5$ and $l(G) \geqs 11$. Similarly, $\l(G) \leqs 9$ and $l(G) \geqs 13$ if $p \geqs 5$. The result follows.

Now let us assume $G = {\rm L}_{r+1}(q)$ and $r \in \{2,4,6,8\}$. First assume $p$ is odd. There is an unrefinable chain ${\rm L}_{r+1}(p) > {\rm PSO}_{r+1}(p) > \O_{r+1}(p)$, so $\l({\rm L}_{r+1}(p)) \leqs 12$ and thus $\l(G) \leqs 2\O(f)+12$. Now $l(G) \geqs fr(r+1)/2+2r-1$ and the desired bound follows if $r>2$. Now assume $G = {\rm L}_{3}(q)$. Since $\O_3(p) \cong {\rm L}_{2}(p)$ we deduce that $\l({\rm L}_{3}(p)) \leqs 6$, so $\l(G) \leqs 2\O(f)+6$ and one checks that the bound $l(G) \geqs 3f+3$ is good enough if $f>2$. If $G = {\rm L}_{3}(p^2)$ then $l(G) \geqs 11$ and $\l(G) \leqs 8$ since there is an unrefinable chain
\[
{\rm L}_{3}(p^2) > {\rm L}_{3}(p).2 > {\rm L}_{3}(p) >  {\rm PSO}_{3}(p) > \O_3(p).
\]
Similarly, if $G = {\rm L}_{3}(p)$ then $\l(G) \leqs 6$ and we note that $l(G) \geqs 10$ if $p \geqs 5$ (this follows from Theorem \ref{t:length}). Finally, if $G = {\rm L}_{3}(3)$ then ${\rm cr}(G) = 8/3$.

To complete the analysis of case (e), let us assume $r \in \{2,4,6,8\}$ and $p=2$. Here $l(G) \geqs fr(r+1)/2+r$ and $\l(G) \leqs 2\O(f)+\l({\rm L}_{r+1}(2))$. As noted in the proof of \cite[Theorem 4]{BLS1}, there is an unrefinable chain ${\rm L}_{r+1}(2) > 2^r.{\rm L}_{r}(2) > {\rm L}_{r}(2)$ and one can check that $\l({\rm L}_{r}(2)) \leqs 5$, so $\l(G) \leqs 2\O(f)+7$. This gives the desired bound unless $r=2$ and $f \leqs 3$. We can exclude the case $f=1$ since ${\rm L}_{3}(2) \cong {\rm L}_{2}(7)$. For $f \in \{2,3\}$ we get $\l(G) \leqs 4$, $l(G) \geqs 9$ and the result follows.

To complete the proof of the lemma, it remains to handle the unitary groups of dimension at most $9$ arising in cases (f) and (g). First consider (f), so $G = {\rm U}_{2r}(q)$, $r \in \{2,3,4\}$ and
\[
|B| = \frac{q^{r(2r-1)}(q^2-1)^r}{(2r,q+1)}.
\]
By arguing as in the proof of \cite[Theorem 4]{BLS1}, we see that $\l(G) \leqs \l({\rm PSp}_{2r}(q))+2$.
If $p=2$, it follows that
\[
\l(G) \leqs \O(f)+2+\l({\rm Sp}_{2r}(2)) \leqs \O(f)+12
\]
(recall that $\l({\rm Sp}_{2r}(2)) \leqs 10$). In view of \eqref{e:bd} we have $l(G) \geqs fr(2r-1)+2r-1$ and one checks that these bounds are sufficient unless $r=2$ and $f=1,2$. Here we compute $\l({\rm U}_{4}(4)) = \l({\rm U}_{4}(2))=5$ and the result follows. Now assume $p>2$. Here
\[
\l(G) \leqs 2\O(f)+2+\l({\rm PSp}_{2r}(p)) \leqs 2\O(f)+15
\]
and \eqref{e:bd} gives $l(G) \geqs fr(2r-1)+4r-2$. These estimates give the result, unless $r=2$ and $f=1,2$. There is an unrefinable chain
\[
{\rm U}_{4}(p^2) > {\rm PSp}_{4}(p^2).2 > {\rm PSp}_{4}(p^2) > {\rm L}_{2}(p^2) > {\rm L}_{2}(p).2 > {\rm L}_{2}(p)
\]
and thus $\l(G) \leqs 9$ for $G = {\rm U}_{4}(p^2)$. Similarly, one checks that $\l(G) \leqs 9$ if $G = {\rm U}_{4}(p)$. In both cases $l(G) \geqs 12$ and the result follows.

Finally, let us consider case (g), where $G = {\rm U}_{2r+1}(q)$, $r \in \{1,2,3,4\}$ and
\[
|B| = \frac{q^{r(2r+1)}(q^2-1)^r}{(2r+1,q+1)}.
\]
First assume $p>2$. Here
\[
\l(G) \leqs \l(\O_{2r+1}(q)) + 2 \leqs 2\O(f)+2+\l(\O_{2r+1}(p)) \leqs 2\O(f)+12
\]
and \eqref{e:bd} gives $l(G) \geqs fr(2r+1)+4r-2$. One checks that these bounds are sufficient unless $r=1$ and $f \leqs 5$. Suppose $G = {\rm U}_{3}(p^f)$ with $f \leqs 5$. If $f =3$ or $5$ then $\l(G) \leqs 2+\l({\rm U}_{3}(p)) \leqs 8$ and the result follows since $l(G) \geqs 11$. If $f=4$ then $l(G) \geqs 14$ and there is an unrefinable chain
\[
{\rm U}_{3}(p^4) > {\rm PSO}_{3}(p^4) > \O_3(p^4) > {\rm L}_{2}(p^2).2 > {\rm L}_{2}(p^2) >  {\rm L}_{2}(p).2 > {\rm L}_{2}(p),
\]
so $\l(G) \leqs 10$. Similarly, if $f=2$ then $\l(G) \leqs 8$ and $l(G) \geqs 11$. Finally, suppose $G = {\rm U}_{3}(p)$. If $p \geqs 7$ then ${\rm U}_{3}(p)> {\rm PSO}_{3}(p)> \O_3(p)$ is unrefinable, so $\l(G) \leqs 6$. It is easy to check that the same bound holds when $p=3$ or $5$, and the desired result now follows since ${\rm cd}(G) \geqs 2$.

Now suppose $p=2$. If $f$ is odd then by considering a chain of subfield subgroups we get
$\l(G) \leqs 2\O(f)+\l({\rm U}_{2r+1}(2))$ and one checks that $\l({\rm U}_{2r+1}(2)) \leqs 6$. For example, ${\rm J}_{3}<{\rm U}_{9}(2)$ is maximal and $\l({\rm J}_{3})=5$, so $\l({\rm U}_{9}(2)) \leqs 6$. Therefore, $\l(G) \leqs 2\O(f)+6$. Since $l(G) \geqs fr(2r+1)+2r-1$, the result follows unless $r=1$ and $f=3$ (note that $(r,f) \ne (1,1)$ since ${\rm U}_{3}(2)$ is soluble). A routine computation gives ${\rm cr}({\rm U}_{3}(8)) = 4$.

Finally, suppose $p=2$ and $f$ is even. Here $l(G) \geqs fr(2r+1)+3r-1$ and we recall that $\l(G) \leqs 3\O(f)+2f+35$. These bounds are sufficient unless $(r,f) = (3,2)$, or $r=2$ and $f \in \{2,4,6\}$, or if $r=1$. If $(r,f) = (3,2)$ then $G = {\rm U}_{7}(4)$ has a maximal subgroup $3277{:}7$ of depth $3$, so $\l(G) \leqs 4$ and the result follows. Similarly, if $r=2$ and $f \in \{2,4,6\}$ then by considering an unrefinable chain through the maximal subgroup
\[
\frac{(q^5+1)}{(q+1)(5,q+1)}{:}5,
\]
we deduce that $\l(G) \leqs 5$ and the result follows since $l(G) \geqs 10f+5$. Finally, let us assume $G = {\rm U}_{3}(2^f)$ with $f$ even. Now $G$ has a reducible maximal subgroup of the form
\[
\frac{q+1}{(3,q+1)}.{\rm L}_{2}(q)
\]
and thus
\[
\l(G) \leqs \l({\rm L}_{2}(q)) +\O(q+1) - \O((3,q+1)) + 1 \leqs  \O(q^2-1)+\O(f)+2 - \O((3,q+1))
\]
since $\l({\rm L}_{2}(q)) \leqs \O(q-1)+\O(f)+1$ by \cite[Theorem 4]{BLS1}.
We also have
\[
l(G) = \O(q^2-1)+3f+1- \O((3,q+1))
\]
by \cite[Theorem 1]{ST}, and one checks that
\[
\O(q^2-1)+3f+1- \O((3,q+1)) > \frac{5}{4}(\O(q^2-1)+\O(f)+2 - \O((3,q+1)))
\]
if $\O(q^2-1)+6 < 7f$. Since $q=2^f$ we have $\O(q^2-1)+6 < 2f+6$ and the result follows.
\end{proof}

\vs

This completes the proof of Theorem \ref{t:maincr}. Notice that Corollary \ref{c:cr} follows immediately. Indeed, we have $l(G) \leqs a\, \cd(G)$ if and only if ${\rm cr}(G) \geqs a/(a-1)$, so Theorem \ref{t:maincr} implies that $a=5$ is the best possible constant.

\subsection{Proof of Theorem \ref{t:cd}}

We begin by recording some immediate consequences of Lemma \ref{l:bwz}.

\begin{prop}\label{p:cd2}
Let $G$ be a finite group.
\begin{itemize}\addtolength{\itemsep}{0.2\baselineskip}
\item[{\rm (i)}] If $1 = G_m \lhd G_{m-1} \lhd \cdots \lhd G_1 \lhd G_0 = G$ is a chain of subgroups
of $G$, then $\cd(G) \geqs \sum_i \cd(G_{i-1}/G_i)$.
\item[{\rm (ii)}] If $G = G_1 \times \cdots \times G_m$, then
$\cd(G) \geqs \sum_i \cd(G_i)$.
\item[{\rm (iii)}] If $T_1, \ldots , T_m$ are the composition factors of $G$, listed with multiplicities, then $\cd(G) \geqs \sum_i \cd(T_i)$.
\end{itemize}
\end{prop}

Note that in part (iii) above we have $\cd(T_i)=0$ if $T_i$ is abelian, so only the
non-abelian composition factors contribute to $\sum_{i} \cd(T_i)$.
Now let $T_1, \ldots , T_m$ be the non-abelian composition factors of $G$ (listed with
multiplicities), and let
\[
\ss(G) = \prod_{i=1}^m T_i
\]
be their direct product.

\begin{prop}\label{p:ss}
We have $l(\ss(G)) \leqs 5\,\cd(G)$ for every finite group $G$.
\end{prop}

\begin{proof}
With the above notation we have
\[
l(\ss(G)) = \sum_{i=1}^m l(T_i).
\]
By Corollary \ref{c:cr} we have $l(T_i) \leqs 5\, \cd(T_i)$ for each $i$, and by combining this with Proposition \ref{p:cd2}(iii) we obtain
\[
l(\ss(G)) \leqs \sum_{i=1}^m 5\, \cd(T_i) = 5\, \sum_{i=1}^m \cd(T_i) \leqs 5\, \cd(G)
\]
as required.
\end{proof}

By a semisimple group we mean a direct product of (non-abelian) finite simple groups.

\begin{lem}\label{l:aut}
If $G$ is a finite semisimple group, then $l({\rm Aut}(G)) \leqs 2l(G)$.
\end{lem}

\begin{proof}
Write $G =\prod_{i=1}^m T_i^{k_i}$ where the $T_i$ are pairwise non-isomorphic finite
(non-abelian) simple groups and $k_i \geqs 1$. Then
\[
{\rm Aut}(G) \cong \prod_{i=1}^m {\rm Aut}(T_i^{k_i}) \cong \prod_{i=1}^m {\rm Aut}(T_i) \wr S_{k_i}.
\]
Hence ${\rm Out}(G) \cong \prod_{i=1}^m {\rm Out}(T_i) \wr S_{k_i}$, so
\[
l({\rm Out}(G)) = \sum_{i=1}^m k_i l({\rm Out}(T_i))+ l(S_{k_i}) \leqs  \sum_{i=1}^m k_i (\log_2 |{\rm Out}(T_i)| + 3/2),
\]
where the last inequality follows from the main theorem of \cite{CST} on the length of the symmetric group. Using the well known orders of ${\rm Out}(T)$ for the finite simple groups $T$ (for example, see \cite{KL}, pp. 170-171), it is easy to verify that $\log_2 |{\rm Out}(T_i)|+ 3/2 \leqs l(T_i)$ for all $i$.
We conclude that
\[
l({\rm Out}(G))\leqs \sum_{i=1}^m k_i l(T_i) = l(G)
\]
and thus $l({\rm Aut}(G)) \leqs 2l(G)$ as required.
\end{proof}

We are now ready to prove Theorem \ref{t:cd}. Let $R(G)$ be the soluble radical of $G$
and consider the semisimple group ${\rm Soc}(G/R(G))$. Applying Proposition \ref{p:ss} to the
group $G/R(G)$ we obtain
\[
l({\rm Soc}(G/R(G))) \leqs l(\ss(G/R(G))) \leqs 5\, \cd(G/R(G)).
\]
It is well known that $G/R(G) \leqs {\rm Aut}({\rm Soc}(G/R(G)))$. Applying the inequality above
with Lemma \ref{l:aut} we obtain
\[
l(G/R(G)) \leqs l({\rm Aut}({\rm Soc}(G/R(G)))) \leqs 2l({\rm Soc}(G/R(G))) \leqs  10 \, \cd(G/R(G)) \leqs 10\, \cd(G)
\]
and the result follows.

\vs

This completes the proof of Theorem \ref{t:cd}.

\appendix

\section{On the number of prime divisors of $p \pm 1$\\
by D.R. Heath-Brown}\label{appendix}

In this appendix, we prove the following result.

\begin{thm}\label{t:hr}
  There are infinitely many primes $p\equiv 5 \imod{72}$ for which
  \[\Omega((p^2-1)/24))\leqs 7.\]
  Hence there are infinitely many primes
  $p$ for which
  \[\max\{\Omega(p\pm 1)\} \leqs 8.\]
\end{thm}

We begin by showing how the second claim follows from the first.
For any prime $p\equiv 5 \imod{72}$ one has $24|(p^2-1)$.  Indeed for such
primes one has $(p-1,72)=4$ and $(p+1,72)=6$. One necessarily
has $\Omega((p-1)/4)\geqs 1$ when $p>5$, so that if $\Omega((p^2-1)/24))\leqs 7$
one must have $\Omega((p+1)/6))\leqs 6$.  It then follows that
$\Omega(p+1)\leqs 8$. The proof that $\Omega(p-1)\leqs 8$ is similar.

To handle the first statement of the theorem we use sieve
methods, as described in the book by Halberstam and Richert \cite{HR},
and in particular the weighted sieve, as in \cite[Chapter 10]{HR}. To
be specific, we apply \cite[Theorem 10.2]{HR} to the set
\[\cA=\{(p^2-1)/24 \,:\, p\equiv 5 \imod{72},\,p\leqs x\}\]
and the set $\mathfrak{P}$ of all primes. The expected value of
\[|\cA_d|:=\#\{n\in\cA\,:\, d|n\}\]
is $X\omega(d)/d$, with $X={\rm Li}(x)/24$, and where $\omega(d)$ is a
multiplicative function satisfying
\[\omega(p)=
\left\{\begin{array}{ll} 0 & \mbox{if $p=2,3$,} \\ 2 & \mbox{if $p \geqs
5$.} \end{array}\right.\]
Condition $(\Omega_1)$, see \cite[p.29]{HR}, is then satisfied with
$A_1=2$, while condition $(\Omega_2^*(\kappa))$, see \cite[p.252]{HR}, holds
with $\kappa=2$ and a suitable numerical constant $A_2$. Moreover
$|\cA_{p^2}|=O(xp^{-2})$, which shows that condition $(\Omega_3)$,
see \cite[p.253]{HR}, also holds, for an appropriate numerical constant
$A_3$.  Finally we consider the condition $(\Omega(R(2,\alpha))$
given in \cite[p.219]{HR}. The primes $p\equiv 5 \imod{72}$ for which
$d$ divides $(p^2-1)/24$ fall into $\omega(d)$ residue classes modulo $72d$,
so that $(\Omega(R(2,\tfrac12))$ holds by an appropriate form of the
Bombieri--Vinogradov theorem, as in \cite[Lemma 3.5]{HR}. This
verifies all the necessary conditions for Theorem 10.2 of \cite{HR},
and the inequality (2.2) of \cite[p.278]{HR} is satisfied (with
$\alpha=\tfrac12$) for any constant $\mu>4$, if $x$ is large enough.

Theorem 10.2 of \cite{HR} then tells us that there are
$\gg X(\log X)^{-2}$ elements $n\in\cA$ which are ``$P_r$-numbers''
(that is to say, one has $\Omega(n)\leqs r$), provided that
\[r>2u-1+\frac{2\int_u^v \frac{1}{\sigma_2(v(\alpha-1/t))}
    \left(1-\frac{u}{t}\right)\frac{dt}{t}}{1-\eta_2(\alpha v)}.\]
Finally, we refer to the calculations of Porter \cite{Por}, and in
particular the last 3 lines of \cite[p.420]{Por}, according to
which it will suffice to have $r>6.7$ if one takes $u=2.2$ and $v=22$.
Since we then have $\alpha^{-1}<u<v$ and $\alpha v=11>\nu_2=4.42\ldots$, by
Porter \cite[Table 2]{Por}, the final conditions (2.3) of
\cite[Theorem 10.2]{HR} are satisfied, and our theorem follows.

\section*{Acknowledgements}

We thank two anonymous referees for their careful reading of the paper and for many helpful comments and suggestions. The third author acknowledges the hospitality of Imperial College, London, while part of this work was carried out. He also acknowledges the support of ISF grant 686/17 and the Vinik chair of mathematics which he holds.

\end{document}